\documentclass[envcountsect,referee]{svjour3}

\smartqed  % flush right qed marks, e.g. at end of proof
\usepackage{graphicx}
\usepackage{latexsym,bm,amsmath,amssymb,mathrsfs}
\usepackage{rotating}
\usepackage{multirow,booktabs,cases}
\usepackage[misc]{ifsym}
\usepackage[style=1]{mdframed}
\usepackage{algorithm,algorithmic}
\usepackage[colorlinks=true]{hyperref}
\hypersetup{urlcolor=blue,citecolor=blue,linkcolor=blue}
\usepackage{epstopdf}

\def\[{\begin{equation}}
\def\]{\end{equation}}
\def\lb{\left(}
\def\rb{\right)}

\def\nn{\nonumber}
\def\t{\top}
\def\A{{\mathcal A}}
\def\Bc{{\mathcal B}}
\def\Cc{{\mathcal C}}
\def\I{{\mathcal I}}
\def\R{{\mathbb R}}

\def\bzeta{{\bm \zeta}}
\def\btau{{\bm \tau}}

\def\btheta{{\bm \vartheta}}
\def\bPhi{{\bm\Phi}}
\def\bUpsilon{{\bm\Upsilon}}
\def\L{{\mathscr L}}
\def\C{{\mathscr C}}
\def\B{{\mathscr B}}

\numberwithin{equation}{section}
\allowdisplaybreaks

\begin{document}
\graphicspath{{./PIC/}}

\title{Higher-degree eigenvalue complementarity problems for tensors}

%\titlerunning{Relaxed projection method}     % if too long for running head

\author{Chen Ling \and Hongjin He \and Liqun Qi}

%\authorrunning{F. Author \and S. Author} % if too long for running head

\institute{C. Ling \and H. He \at
Department of Mathematics, School of Science, Hangzhou Dianzi University, Hangzhou, 310018, China.\\
\email{macling@hdu.edu.cn}
\and H. He \at
\email{hehjmath@hdu.edu.cn}
\and L. Qi \at
Department of Applied Mathematics, The Hong Kong Polytechnic University, Hung Hom, Kowloon, Hong Kong. \\
\email{maqilq@polyu.edu.hk}
 }

\date{Received: date / Accepted: date}

\maketitle

\begin{abstract}
In this paper, we introduce a unified framework of Tensor Higher-Degree Eigenvalue Complementarity Problem (THDEiCP), which goes beyond the framework of the typical Quadratic Eigenvalue Complementarity Problem (QEiCP) for matrices. First, we study some topological properties of higher-degree cone eigenvalues of tensors. Based upon the symmetry assumptions on the underlying tensors, we then reformulate THDEiCP as a weakly coupled homogeneous polynomial optimization problem, which might be greatly helpful for designing implementable algorithms to solve the problem under consideration numerically. As more general theoretical results, we present the results concerning existence of solutions of THDEiCP without symmetry conditions.
Finally, we propose an easily implementable algorithm to solve THDEiCP, and report some computational results.




\end{abstract}

\section{Introduction}\label{Sec1}
Let $A$, $B$, $C\in \R^{n\times n}$ be given matrices, the generic Quadratic Eigenvalue Complementarity Problem (QEiCP) widely studied in recent papers, e.g., see \cite{BIJ14,FJFI14,S01}, is that of finding $(\lambda,x)\in \mathbb{R}\times \mathbb{R}^n$ such that
\begin{equation}\label{QMEiCP}
\left\{
\begin{array}{l}
K\ni x \bot (\lambda^2A+\lambda B+C)x\in K^*,\\
e^\top x=1,
\end{array}
\right.
\end{equation}
where $e=(1,1,\ldots, 1)^\top \in \mathbb{R}^n$, $K$ is a closed convex cone in $\mathbb{R}^n$, and $K^*$ refers to the dual cone of $K$, which is defined by
$$
K^*:=\left\{y\in \mathbb{R}^n~|~\langle y,x\rangle \geq 0,~\forall~x\in K\right\}.
$$
Without loss of generality, the linear constraint $e^\t x=1$ in \eqref{QMEiCP} plays an important role in preventing the $x$ component of a solution to vanish. Notice that the leading matrix $A$ could be singular, and in particular, the QEiCP immediately reduces to the classical Eigenvalue Complementarity Problem (EiCP) when $A=0$. Clearly, QEiCP is an interesting generalization of the classical EiCP with embracing an extra quadratic term on $\lambda$. When $K$ is taken as the whole space $\mathbb{R}^n$, then \eqref{QMEiCP} becomes the well studied unconstrained quadratic eigenvalue problem, which frequently arises in areas such as the electric power systems \cite{MR06}, the dynamic analysis of acoustic systems \cite{BDRS00} and linear stability of flows in fluid mechanics \cite{MT00}, to name a few. We refer the reader to \cite{TM01} for a survey on the unconstrained version. If the matrices $A,B,C$ are all symmetric, QEiCP and EiCP are called symmetric, respectively.

Since the seminal work on EiCP \cite{CFJM01} devoted to the study of static equilibrium states of mechanical systems with unilateral friction, both EiCP and QEiCP have been well discussed from theoretical and numerical perspective in the literature, e.g., see \cite{AS13,PS11,FJFI14,FJSF14,FJSFu14,JRRS08,JSR07,JSRR09,S01}, where these papers only focus on matrix cases. For instance, in order to study the  sufficient conditions for the existence of solutions of QEiCP, the so-called {\it co-regularity} and {\it co-hiperbolicity} properties were introduced in \cite{S01}. Usually, the {\it co-regularity} on matrix $A\in\R^{n\times n}$, i.e., $x^\top Ax\neq0$ for any $x\in K$, means that $A$ or $-A$ is strictly $K$-positive. We say that $(A,B,C)\in{\mathcal M}_n:=\R^{n\times n}\times\R^{n\times n}\times\R^{n\times n}$ satisfies {\it co-hiperbolicity} property, if
$$
(x^\top Bx)^2\geq 4(x^\top Ax)(x^\top Cx),~~~~\forall~x\in K.
$$
However, checking whether a given matrices triplet $(A,B,C)\in{\mathcal M}_n$ satisfies co-hyperbolicity or not is co-NP-complete, which is essentially the verification problem of {\it copositiveness} (see Definition \ref{copositivedef}) of a related fourth order $n$-dimensional tensor \cite{Q13,SQ14}.

In recent decade, tensor, which is a natural extension of the concept of matrices, is on the timely topic of high-dimensional data representation in terms of theoretical analysis and algorithmic design because of its widespread applications in engineering. A tensor, namely, is a multidimensional array, whose order is the number of dimensions. Let $m$ and $n$ be positive integers. We call $\mathcal{A}= (a_{i_1\ldots i_m} )$, where $a_{i_1\ldots i_m} \in \mathbb{R}$ for $1\leq i_1, \ldots, i_m\leq n$, a real $m$-th
order $n$-dimensional real square tensor. For the sake of convenience, we denote by $\mathcal{T}_{m,n}$ the space of $m$-th order $n$-dimensional real square tensors. Furthermore, a tensor $\mathcal{A}\in \mathcal{T}_{m,n}$ is called {\it symmetric} if its entries are invariant under any permutation of its indices.
For a vector $x = (x_1, \ldots,x_n)^\top\in \mathbb{C}^n$ and a tensor $\mathcal{A}= (a_{i_1\ldots i_m})\in  \mathcal{T}_{m,n}$,  $\mathcal{A}x^{m-1}$ is an $n$-dimensional vector with its $i$-th component defined by
$$
(\mathcal{A}x^{m-1})_i=\sum_{i_2,\ldots,i_m=1}^na_{ii_2\ldots i_m}x_{i_2}\cdots x_{i_m},~~{\rm for~}i=1,2,\ldots,n,$$
and $\mathcal{A}x^m$ is the value at $x$ of a homogeneous polynomial, defined by
$$\mathcal{A}x^{m}=\sum_{i_1,i_2,\ldots,i_m=1}^na_{i_1i_2\ldots i_m}x_{i_1}x_{i_2}\cdots x_{i_m}.$$

Although tensor-related problems have been received considerable attention many years ago, the history of research on eigenvalues (eigenvectors) of a square tensor can be traced back to the pioneer works independently introduced by Qi \cite{Q05} and Lim \cite{Lim05}. Comparatively speaking, the developments of eigenvalue-related problems for tensors are still in their infancy. For given tenors $\mathcal{A},\mathcal{B}\in \mathcal{T}_{m,n}$, we say that $(\mathcal{A},\mathcal{B})$ is an {\it identical singular} pair, if
$$
\left\{x\in \mathbb{C}^n\backslash\{0\}~|~\mathcal{A}x^{m-1}=0,\;\mathcal{B}x^{m-1}=0\right\}\neq\emptyset.
$$
When assuming that $(\mathcal{A},\mathcal{B})$ is not an identical singular pair, we
say $(\lambda, x)\in \mathbb{C}\times(\mathbb{C}^n\backslash\{0\})$ is an eigenvalue-eigenvector pair of $(\mathcal{A}, \mathcal{B})$, if the following $n$-system of equations
\begin{equation}\label{ABEigen}
(\mathcal{A}-\lambda \mathcal{B})x^{m-1}=0
\end{equation}
possesses a {\it nonzero} solution $x$. This unified definition of eigenvalue-eigenvector pair for tensors was introduced by Chang et al. \cite{CPZ09}.  In recent years, it is well documented in the literature that tensors and eigenvalues/eigenvectors of tensors have fruitful applications in various fields such as magnetic resonance imaging \cite{BV08,QYW10}, higher-order Markov chains \cite{NQZ09} and best-rank one approximation in data analysis \cite{QWW09}, whereby many nice properties such as the Perron-Frobenius theorem for eigenvalues/eigenvectors of nonnegative square tensors have been well established, see, e.g., \cite{CPZ08,YY10}. All these encourage us to consider tensor eigenvalue complementarity problems. However, to the best of our knowledge, the most recent paper \cite{LHQ15} is the first work devoted to the Tensor Generalized Eigenvalue Complementarity Problem (TGEiCP), whereas leaving {\it higher-degree} cases a big gap. Therefore, the main objective of this paper is to fill out this gap.

In this paper, we consider the Tensor Higher-Degree Eigenvalue Complementarity Problem (THDEiCP), which goes beyond the framework of QEiCP and further generalizes TGEiCP. Mathematically, the THDEiCP can be characterized as finding a scalar $\lambda\in \mathbb{R}$ and a vector $x\in \mathbb{R}^n\backslash\{0\}$ such that
\begin{equation}\label{QTECiP}
K\ni  x\perp( \lambda ^m\mathcal{A}+\lambda\mathcal{B}+\mathcal{C}) x^{m-1}\in K^*,
\end{equation}
where $\mathcal{A}=(a_{i_1i_2\ldots i_m})$, $\mathcal{B}=(b_{i_1i_2\ldots i_m})$, and $\mathcal{C}=(c_{i_1i_2\ldots i_m})\in \mathcal{T}_{m,n}$. Correspondingly, the scalar $\lambda$ and the nonzero vector $x$ satisfying system \eqref{QTECiP} are respectively called an $m$-degree $K$-eigenvalue
of the tensors triplet $\mathcal{Q}:=(\mathcal{A, B, C})\in \mathcal{F}_{m,n}:=\mathcal{T}_{m,n}\times \mathcal{T}_{m,n}\times \mathcal{T}_{m,n}$ and an associated $K$-eigenvector. Alternatively, $(\lambda, x)$ is also called an $m$-degree $K$-eigenpair of $\mathcal{Q}$. Throughout, the set of all $m$-degree $K$-eigenvalues of $\mathcal{Q}$ is called the $m$-degree $K$-spectrum of $\mathcal{Q}$, i.e.,
\begin{equation}\label{sigmaK}
\sigma(\mathcal{Q},K):=\left\{\lambda\in \mathbb{R}~|~\exists~ x\in \mathbb{R}^n\backslash\{0\}, ~K\ni x\perp (\lambda ^m\mathcal{A} + \lambda\mathcal{B}+\mathcal{C} )x^{m-1}\in K^*\right\}.
\end{equation}
If $K:=\mathbb{R}_+^n$, the $m$-degree $K$-eigenvalue/eigenvector of the tensors triplet $\mathcal{Q}$ is called the $m$-degree Pareto-eigenvalue/eigenvector of $\mathcal{Q}$, and the $m$-degree $K$-spectrum of $\mathcal{Q}$ is called the $m$-degree Pareto-spectrum of $\mathcal{Q}$. If $x\in {\rm int}(K)$ (resp. $x\in \mathbb{R}_{++}^n)$, then $\lambda$ is called a strict $m$-degree $K$-eigenvalue (resp. Pareto-eigenvalue) of $\mathcal{Q}$.

It is clear from \eqref{QTECiP} that THDEiCP covers TGEiCP and QEiCP as the special cases. More concretely, by taking $\A=0$, model \eqref{QTECiP} immediately reduces to TGEiCP studied in \cite{LHQ15}. When we set $m=2$, THDEiCP clearly corresponds to the QEiCP \eqref{QMEiCP}. Like \cite{Se99}, on the other hand, we can further establish the connection between THDEiCP and a class of differential inclusions with nonconvex processes $\Gamma$ defined by
$$
Gr(\Gamma):=\{(x,y)\in \mathbb{R}^n\times \mathbb{R}^n~|~K\ni x\bot(\mathcal{A}y^{m-1}-\mathcal{B}x^{m-1})\in K^* \}.
$$
Accordingly, for the differential inclusions defined by $\dot{{\bm u}}(t)\in \Gamma({\bm u}(t))$, as noticed already by Rockafellar \cite{Ro79}, the change of variables ${\bm u}(t)={\rm exp}(\lambda t)x$ with $\lambda>0$ leads to the equivalent system $\lambda x\in \Gamma(x)$. Therefore, if the pair $(\lambda, x)$ satisfies $\lambda x\in \Gamma(x)$, then the trajectory $t\mapsto {\rm exp}(\lambda t)x$ is a solution to the considered differential inclusions. Moreover,
if the trajectory constructed above is nonconstant, then $x$ must be a nonzero vector; this requires
$(\lambda, x)$ to be a solution of THDEiCP with $\mathcal{C}=0$ because of $\lambda >0$.

The paper is divided into six sections. As far as we know, it might be the first work on THDEiCP, we thus do not know whether the topological properties of QEiCP still hold for the newly introduced model. In Section \ref{TopProp}, we first study, as briefly as possible, some topological properties such as {\it closedness}, {\it boundedness}, and {\it upper-semicontinuity} of the $m$-degree $K$-spectrum given by \eqref{sigmaK}, in addition to estimating upper bounds on the number of eigenvalues of THDEiCP. With the preparations on these topological properties, ones may be further concerned with that how to solve the model under consideration. In Section \ref{OptForm}, we reformulate the special case of THDEiCP with symmetric $\A$ and $\Bc$ as a {\it weakly coupled} polynomial optimization problem for the case where $K:=\mathbb{R}_+^n$, which would potentially facilitate the algorithmic design. From a theoretical perspective, in Section \ref{Existence}, we establish the results concerning existence of solutions of THDEiCP without symmetry assumptions on $\A$ and $\Bc$. Based upon the augmented Lagrangian method, in Section \ref{NumAlg}, we propose an implementable splitting algorithm to solve the resulting polynomial optimization reformulation of the symmetric THDEiCP and report some preliminary results. Finally, we give some concluding remarks in Section \ref{Concl}.

\medskip

\noindent{\bf Notation.} Let $\mathbb{R}^n$ denote the real Euclidean space of column vectors of length $n$, which is equipped with the standard inner product $\langle y,x\rangle=y^\top x$ and the associated
norm. The superscript `$\top $' indicates transposition and the symbol `$\bot$' represents orthogonality. Denote $\mathbb{R}_+^n=\{x\in \mathbb{R}^n~|~x\geq 0\}$ and $\mathbb{R}_{++}^n=\{x\in \mathbb{R}^n~|~x> 0\}$. For a vector $x\in \mathbb{R}^n$  and a subset $J$ of the index set $[n]:=\{1,2,\ldots,n\}$, we use the notation $x_J$ for the $|J|$ dimensional sub-vector of $x$, which is obtained by deleting all components $i\not\in J$, where the symbol $|J|$ denotes the cardinality of $J$. For a vector $x\in \mathbb{R}^n$ and an integer $r\geq 0$, denote $x^{[r]}=(x_1^r,x_2^r,\ldots, x_n^r)^\top$, and denote by ${\rm diag}(x)$ the $n\times n$ diagonal matrix containing $x_i$ in its diagonal. Moreover, for $\mathcal{A}\in \mathcal{T}_{m,n}$, we denote the principal sub-tensor of $\mathcal{A}$ by $\mathcal{A}_J$, which is obtained by homogeneous polynomial $\mathcal{A}x^m$ for all $x=(x_1,x_2,\ldots,x_n)^\top$ with $x_i=0$ for $[n]\backslash J$. So, $\mathcal{A}_J\in \mathcal{T}_{m,|J|}$. Denote by $e\in \mathbb{R}^n$ with all entries being $1$, i.e., $e=(1,1,\ldots,1)^\top$. Denote by $\mathcal{I}=(\delta_{i_1\ldots i_m})$ the unit tensor in $\mathcal{T}_{m,n}$, where $\delta_{i_1\ldots i_m}$ is the Kronecker symbol
$$
\delta_{i_1\ldots i_m}=\left\{
\begin{array}{ll}
1,&\;\;{\rm if~}i_1=\ldots =i_m,\\
0,&\;\;{\rm otherwise}.
\end{array}
\right.
$$

\section{Some basic properties of $K$-spectrum}\label{TopProp}

In this section, we summarize some basic definitions and study some basic topological properties of $m$-degree cone spectrum, which will be used in subsequent sections.

We first give the concept of cone positive square tensors, which is a generalized concept of copositive square tensor introduced in \cite{Q13} and studied in \cite{SQ14}.

\begin{definition}\label{copositivedef}
Let $K$ be a closed convex cone in $\mathbb{R}^n$ and $\mathcal{G}\in \mathcal{T}_{m,n}$. We say that $\mathcal{G}$ is a (resp. strictly) $K$-positive tensor, if $\mathcal{G}x^m\geq 0$ (resp. $>0$) for any $x\in K\backslash\{0\}$. If $K=\mathbb{R}_+^n$, the (strictly) $K$-positive tensor $\mathcal{G}$ is said the (strictly) copositive tensor.
\end{definition}

It is obvious from the notation of $\mathcal{F}_{m,n}$ that $\mathcal{F}_{m,n}$ is a linear space. The distance between two elements $\mathcal{Q}_i=(\mathcal{A}_i,\mathcal{B}_i,\mathcal{C}_i)\in\mathcal{F}_{m,n}~(i=1,2)$ is measured by means of
the expression
$$
\|\mathcal{Q}_1-\mathcal{Q}_2\|_F=\left\{\|\mathcal{A}_1-\mathcal{A}_2\|_F^2+\|\mathcal{B}_1-\mathcal{B}_2\|_F^2+\|\mathcal{C}_1-\mathcal{C}_2\|_F^2\right\}^{\frac{1}{2}},
$$
where
$$
\|\mathcal{A}\|_F=\sqrt{\sum_{1\leq i_1,\ldots,i_m\leq n}a^2_{i_1\ldots i_m}},~~~\forall~\mathcal{A}=(a_{i_1\ldots i_m})\in \mathcal{T}_{m,n}.
$$

Denote by $\C(\mathbb{R}^n)$ the set of nonzero closed convex cones in $\mathbb{R}^n$, which is associated with the natural metric defined by
$$
\delta(K_1,K_2) := \sup_{\|z\|\leq 1}|{\rm dist}(z,K_1)- {\rm dist}(z,K_2)|,
$$
where ${\rm dist}(z,K) := {\rm inf}_{u\in K}\|z-u\|$ stands for the distance from $z$ to $K$. An equivalent
way of defining $\delta$ is
$$\delta(K_1,K_2) = {\rm haus}(K_1\cap \B_n,K_2 \cap \B_n),$$
where $\B_n$ is the closed unit ball in $\mathbb{R}^n$, and
$$
{\rm haus}(\C_1,\C_2) := \max\left\{\sup_{z\in \C_1}{\rm dist}(z,\C_2), \sup_{z\in \C_2}{\rm dist}(z,\C_1)\right\}
$$
stands for the Hausdorff distance between the compact sets $\C_1,\C_2\subset \mathbb{R}^n$ (see \cite[pp. 85-86]{Be93}). General information on the metric $\delta$ can be consulted in
the book by Rockafellar and Wets \cite{RW98}. According to \cite{WW67}, the
operation $K\mapsto  K^*$ is an isometry on the space $(\C(\mathbb{R}^n), \delta)$, that is to say,
$$
\delta(K_1^*,K_2^*)=\delta(K_1,K_2), ~~~{\rm for~all~} K_1,K_2\in \C(\mathbb{R}^n).
$$
The basic topological properties of the mapping $\sigma : \mathcal{F}_{m,n} \times \C(\mathbb{R}^n)\rightarrow 2^{\mathbb{R}}$, defined in \eqref{sigmaK}, are listed
in the next proposition. This proposition is a tensor version of generalizing the results presented in \cite{S01}. As far as semicontinuity concepts are concerned, we use the
following terminology (cf. Section 6.2 in \cite{Be93}):

\begin{definition}
Let $W$ and $Y$ be two topological spaces. The mapping $\Psi:W\rightarrow 2^Y$ is said to be upper-semicontinuous (resp. lower-semicontinuous) if the set $$
\{w\in W~|~\Psi(w)\subset U\}~~~\left({\rm resp.~~}\{w\in W~|~\Psi(w)\cap U\neq\emptyset\}\right)
$$
is open, whenever $U\subset Y$ is open.
\end{definition}
\begin{definition}
Let $\mathcal{Q}=(\mathcal{A,B,C})\in \mathcal{F}_{m,n}$ and $K\in \C(\mathbb{R}^n)$. We say that $\mathcal{Q}$ is $K$-regular if the leading tensor $\mathcal{A}$ satisfies
$$
\mathcal{A}x^m\neq 0,~~ \forall~ x\in K\backslash \{0\}.
$$
\end{definition}

 It is obvious that, if $\mathcal{Q}$ is $K$-regular, then the leading tensor $\mathcal{A}$ in $\mathcal{Q}$ or $-\mathcal{A}$ is $K$-positive.
\begin{proposition}\label{Propt5}
The following three statements are true:
\begin{itemize}
\itemindent8pt
\item[{\rm(i)}.] The set $\Sigma:=\{(\mathcal{Q},K,\lambda) \in \mathcal{F}_{m,n} \times \C(\mathbb{R}^n)\times \mathbb{R}~|~\lambda\in \sigma(\mathcal{Q},K)\}$ is closed in the product space $\mathcal{F}_{m,n} \times \C(\mathbb{R}^n)\times \mathbb{R}$. In particular,
for any $(\bar{\mathcal{Q}},\bar K)\in \mathcal{F}_{m,n} \times \C(\mathbb{R}^n)$, $\sigma(\bar {\mathcal{Q}},\bar K)$ is a closed subset of $\mathbb{R}$;
\item[{\rm(ii)}.] Let $(\bar{\mathcal{Q}},\bar K)\in \mathcal{F}_{m,n}\times \C(\mathbb{R}^n)$. If $\bar{\mathcal{Q}}$ is $\bar K$-regular, then the mapping $\sigma:\mathcal{F}_{m,n} \times \C(\mathbb{R}^n)\rightarrow 2^\mathbb{R}$ is locally bounded at $(\bar{\mathcal{Q}},\bar K)$, i.e., $\bigcup_{(\mathcal{Q},K)\in \mathcal{N}}\sigma(\mathcal{Q},K)$ is bounded for some neighborhood $\mathcal{N}$ of $(\bar {\mathcal{Q}},\bar K)$.
\item[{\rm(iii)}.] Let $(\bar{\mathcal{Q}},\bar K)\in \mathcal{F}_{m,n}\times \C(\mathbb{R}^n)$. If $\bar{\mathcal{Q}}$ is $\bar K$-regular, then $\sigma$ is upper-semicontinuous at $(\bar{\mathcal{Q}},\bar K)$.
\end{itemize}
\end{proposition}

\begin{proof}
(i). The closedness of $\Sigma$ amounts to saying that
\begin{equation*}\label{DPk}
\left.
\begin{array}{r}
(\mathcal{Q}_\nu, K_\nu)\rightarrow (\bar {\mathcal{Q}},\bar K),~\lambda_\nu\rightarrow \bar\lambda\\
\lambda_\nu\in \sigma(\mathcal{Q}_\nu,K_\nu)
\end{array}
\right\}\Rightarrow \bar\lambda\in  \sigma(\bar{\mathcal{Q}},\bar K).
\end{equation*}
Since $\lambda_\nu\in \sigma(\mathcal{Q}_\nu,K_\nu)
$, there exists a vector $x_\nu\in \mathbb{R}^n\backslash\{0\}$ such that
\begin{equation}\label{nuTSVCP}
K_\nu \ni x_\nu\perp(\lambda^m_\nu \mathcal{A}_\nu +\lambda_\nu \mathcal{B}_\nu+\mathcal{C}_\nu)x_\nu^{m-1}\in K_\nu^*.
\end{equation}
Let $\bar x_\nu=x_\nu/\|x_\nu\|$. From the homogeneity of (\ref{nuTSVCP}) on $x$, we know that
\begin{equation}\label{barnuTSVCP}
K_\nu \ni \bar x_\nu\perp(\lambda^m_\nu \mathcal{A}_\nu +\lambda_\nu \mathcal{B}_\nu+\mathcal{C}_\nu)\bar x_\nu^{m-1}\in K_\nu^*.
\end{equation}It is clear that $\|\bar x_\nu\|=1$ for every $\nu=1,2,\ldots$. Without loss of generality, we assume  that $\bar x_\nu\rightarrow \bar x$. It is obvious that $\|\bar x\|=1$, which means $\bar x\in \mathbb{R}^n\backslash\{0\}$. Since $\delta(K_1^*,K_2^*)=\delta(K_1,K_2)$ for any $K_1,K_2\in \C(\mathbb{R}^n)$, by passing to the limit in (\ref{barnuTSVCP}), one knows
$$
\bar K \ni \bar x\perp(\bar\lambda^m \bar{\mathcal{A}} +\bar\lambda \bar{\mathcal{B}}+\bar{\mathcal{C}})\bar x^{m-1}\in \bar K^*,
$$
which implies $\bar \lambda\in \sigma(\bar {\mathcal{Q}},\bar K)$ due to $\bar x\in \mathbb{R}^n\backslash\{0\}$. We proved the first part (i) of this proposition.

(ii). Suppose that the map $\sigma$ is not locally bounded at $(\bar{\mathcal{Q}},\bar K)$. Then there exists a sequence $\{\mathcal{Q}_\nu,K_\nu,\lambda_\nu\}$ satisfying $$
\|\mathcal{Q}_\nu-\bar{\mathcal{Q}}\|_F\rightarrow 0,~~~\delta(K_\nu,\bar K)\rightarrow 0,~~~{\rm and}~~~|\lambda_\nu|\rightarrow +\infty,
$$
such that $\lambda_\nu\in \sigma(\mathcal{Q}_\nu,K_\nu)$ for any $\nu=1,2,\ldots$. Consequently, there exist vectors $x_\nu\in K_{\nu}$ with $\|x_\nu\|=1$, such that
\begin{equation}\label{nuTSVCP9}
K_\nu \ni x_\nu\perp(\lambda^m_\nu \mathcal{A}_\nu +\lambda_\nu \mathcal{B}_\nu+\mathcal{C}_\nu)x_\nu^{m-1}\in K_\nu^*.
\end{equation}
By (\ref{nuTSVCP9}), we have
$$
(\lambda^m_\nu \mathcal{A}_\nu +\lambda_\nu \mathcal{B}_\nu +\mathcal{C}_\nu )x_\nu^{m}=0,
$$
which implies
$$
\left(\mathcal{A}_\nu +\frac{\mathcal{B}_\nu }{\lambda_\nu^{m-1}} +\frac{\mathcal{C}_\nu }{\lambda_\nu^{m}}\right)x_\nu^{m}=0.
$$
We assume, without loss of generality, that $x_\nu\rightarrow \bar x$. It is obvious that $\|\bar x\|=1$, which means $\bar x\in \mathbb{R}_+^n\backslash\{0\}$. By passing to the limit in the above expression, it holds that $\bar{\mathcal{A}} \bar x^{m}=0$. It contradicts the $\bar K$-regularity of $\bar {\mathcal{Q}}$ by the truth $\bar x\in K\backslash\{0\}$.

(iii). Suppose that $\sigma$ is not upper-semicontinuous at $(\bar{\mathcal{Q}},\bar K)$, then we could find an open set $\bar U\subset \mathbb{R}$ and a sequence $\{(\mathcal{Q}_\nu, K_\nu)\}$ satisfying $(\mathcal{Q}_\nu, K_\nu)\rightarrow (\bar{\mathcal{Q}}, \bar K)$, such that
$$
\sigma(\bar {\mathcal{Q}},\bar K)\subset \bar U~~~{\rm but}~~\sigma(\mathcal{Q}_\nu, K_\nu)\cap (\mathbb{R}\backslash \bar U)\neq\emptyset, ~~~~{\rm for~any}~\nu=1,2,\ldots.
$$
Now, for each $\nu$, pick up $\lambda_\nu\in \sigma(\mathcal{Q}_\nu, K_\nu)\cap (\mathbb{R}\backslash \bar U)$. It follows from (ii) that the sequence $\{\lambda_\nu\}$ admits a converging subsequence. By (i), the corresponding
limit must be in $\sigma(\bar{\mathcal{Q}},  \bar K)\cap (\mathbb{R}\backslash  \bar U)$, which together with  $\sigma(\bar {\mathcal{Q}},\bar K)\subset \bar U$ leads to a contradiction.
\qed
\end{proof}

From the first two parts (i) and (ii) of Proposition \ref{Propt5}, we have the following corollary.

\begin{corollary}\label{compact} Let $(\mathcal{Q},K)\in \mathcal{F}_{m,n}\times \C(\mathbb{R}^n)$. If $\mathcal{Q}$ is $K$-regular, then $\sigma(\mathcal{Q},K)$ is compact.
\end{corollary}

Below, we present a preliminary estimation on the numbers of $m$-degree Pareto-eigenvalues. We first present the following proposition which fully characterizes the $m$-degree Pareto-spectrum of THDEiCP.

\begin{proposition}\label{ECharact} Let $\mathcal{Q}=(\mathcal{A}, \mathcal{B}, \mathcal{C})\in \mathcal{F}_{m,n}$. A real number $\lambda$ is an $m$-degree Pareto-eigenvalue of $\mathcal{Q}$, if and only if there exists a nonempty subset $J\subseteq  [n]$ and a vector $w\in \mathbb{R}_{++}^{|J|}$ such that
\begin{equation}\label{Keigenvalue}
(\lambda^m\mathcal{A}_J+\lambda \mathcal{B}_J+\mathcal{C}_J)w^{m-1}=0
\end{equation}
and
\begin{equation*}\label{KeigenIn}
\sum_{i_2,\ldots,i_m\in J}(\lambda^m a_{ii_2\ldots i_m}+\lambda b_{ii_2\ldots i_m}+c_{ii_2\ldots i_m})w_{i_2}\cdots w_{i_m}\geq 0, ~~~\forall~i\in [n]\backslash J.
\end{equation*}
In such a case, the vector $x\in \mathbb{R}^n_+$ defined by
$$
x_i=\left\{
\begin{array}{ll}
w_i,&i\in J,\\
0,&i\in [n]\backslash J
\end{array}
\right.
$$
is a Pareto-eigenvector of $\mathcal{Q}$, associated to the $m$-degree Pareto-eigenvalue $\lambda$.
\end{proposition}

\begin{proof}
It can be proved in a similar way that used in \cite{SQ13}. \qed
\end{proof}

It is well known that, on the left-hand side of (\ref{Keigenvalue}), $(\lambda^m\mathcal{A}_J+\lambda \mathcal{B}_J+\mathcal{C}_J)w^{m-1}$ is indeed a set of $|J|$ homogeneous polynomials in $|J|$ variables, denoted
by $\{P^{\lambda}_i(w) ~|~ 1\leq i\leq |J|\}$, of degree $(m-1)$. In the complex field, in order to study the solution set of a system of $|J|$ homogeneous
polynomials $(P_1,\ldots,P_{|J|})$, in $|J|$ variables, the concept of the resultant ${\rm Res}(P_1,\ldots, P_{|J|})$ is well defined and introduced, we refer
to \cite{CLS06} for details. Applying to our current problem, ${\rm Res}(P_1,\ldots, P_{|J|})$ has the following properties.
\begin{proposition}\label{Prop1} We have the following results:
\begin{itemize}
\itemindent6pt
\item[{\rm(i)}.] ${\rm Res}(P_1,\ldots, P_{|J|})=0$, if and only if there exists $(\lambda,x)\in \mathbb{C}\times(\mathbb{C}^{|J|}\backslash\{0\})$ satisfying the relation \eqref{Keigenvalue}.
\item[{\rm(ii)}.] The degree of $\lambda$ in ${\rm Res}(P_1,\ldots, P_{|J|})$ is at most $m|J|(m-1)^{|J|-1}$.
\end{itemize}
\end{proposition}

By Proposition \ref{ECharact}, if $\lambda$ is an $m$-degree Pareto-eigenvalue of $\mathcal{Q}=(\mathcal{A},\mathcal{B}, \mathcal{C})\in \mathcal{F}_{m,n}$, then there exists a nonempty subset $J\subseteq  [n]$ such that $\lambda$ is a strict $m$-degree Pareto-eigenvalue of $\mathcal{Q}_J=(\mathcal{A}_J, \mathcal{B}_J, \mathcal{C}_J)\in \mathcal{F}_{m,|J|}$. We now state and prove one of main results in this section.

\begin{theorem}\label{Proposition2} Let  $\mathcal{Q}=(\mathcal{A},\mathcal{B}, \mathcal{C})\in \mathcal{F}_{m,n}$. Assume that $\mathcal{Q}$ is $\mathbb{R}_+^n$-regular. Then, there are at most $\btau_{m,n}:=nm^{n}$ $m$-degree Pareto-eigenvalues of $\mathcal{Q}$.
\end{theorem}
\begin{proof} It is obvious that, for every $k=0,1,\ldots,n-1$, there are $\binom{n}{n-k}$ corresponding principal
sub-tensors triplet of order $m$ dimension $(n-k)$. Moreover, by Proposition \ref{Prop1}, we know that every principal sub-tensors triplet of order $m$ dimension $(n-k)$ can have at most $m(n-k)(m-1)^{n-k-1}$ strict $m$-degree Pareto-eigenvalues. By Proposition \ref{ECharact}, in this
way one obtains the upper bound
$$
\btau_{m,n}=\sum_{k=0}^{n-1}\binom{n}{n-k}m(n-k)(m-1)^{n-k-1}=nm^{n}.
$$
We obtain the desired result.
\qed\end{proof}

Now we extend the above result to the more general case where $K$ is a polyhedral convex cone. A closed convex cone $K$ in $\mathbb{R}^n$ is said to be finitely
generated if there is a linearly independent collection $\{\eta_1,\eta_2,\ldots,\eta_p\}$ of vectors in $\mathbb{R}^n$ such that
\begin{equation}\label{Kdef}
K={\rm cone}\{\eta_1,\eta_2,\ldots,\eta_p\}=\left\{\sum_{i=1}^pa_j\eta_j~|~\alpha=(a_1,a_2,\ldots,a_p)^\top\in \mathbb{R}_+^p\right\}.
\end{equation}
 It is clear that $K=\{H^\top\alpha ~|~\alpha\in \mathbb{R}_+^p\}$, where $H=[\eta_1,\eta_2,\ldots,\eta_p]^\top$. Moreover, it is easy to see that the dual cone $K^*$ of $K$ is equivalent to $\{z\in \mathbb{R}^n~|~Hz\geq 0\}$.
\begin{theorem}\label{proposition4}
Let  $\mathcal{Q}=(\mathcal{A},\mathcal{B}, \mathcal{C})\in \mathcal{F}_{m,n}$. Let $K$ be represented by \eqref{Kdef}. Assume that $\mathcal{Q}$ is $K$-regular. Then, there are at most $\btau_{m,p}:=pm^{p}$ $m$-degree $K$-eigenvalues of $\mathcal{Q}$.
\end{theorem}
\begin{proof} We first prove that problem (\ref{QTECiP}) with $K$ defined by (\ref{Kdef}) is equivalent to finding a vector $\bar \alpha\in \mathbb{R}^p\backslash\{0\}$ and $\bar \lambda\in \mathbb{R}$ such that
\begin{equation}\label{ECPEq}
\bar \alpha \geq 0,~~~(\bar \lambda^m \mathcal{D}+\bar \lambda\mathcal{G}+\mathcal{S})\bar \alpha^{m-1}\geq 0, ~~~\langle \bar \alpha,(\bar \lambda^m \mathcal{D}+\bar \lambda\mathcal{G}+\mathcal{S})\bar \alpha^{m-1}\rangle=0,
\end{equation}
where $\mathcal{D}$, $\mathcal{G}$  and $\mathcal{S}$ are three $m$-th order $p$-dimensional tensors, whose elements are denoted by
$$
d_{i_1i_2\ldots i_m}=\sum_{j_1,j_2,\ldots,j_m=1}^na_{j_1j_2\ldots j_m}h_{i_1j_1}h_{i_2j_2}\cdots h_{i_mj_m},
$$
$$
g_{i_1i_2\ldots i_m}=\sum_{j_1,j_2,\ldots,j_m=1}^nb_{j_1j_2\ldots j_m}h_{i_1j_1}h_{i_2j_2}\cdots h_{i_mj_m},
$$
and
$$
s_{i_1i_2\ldots i_m}=\sum_{j_1,j_2,\ldots,j_m=1}^nc_{j_1j_2\ldots j_m}h_{i_1j_1}h_{i_2j_2}\cdots h_{i_mj_m}
$$
for $1\leq i_1,i_2,\ldots, i_m\leq p$, respectively.

Let $(\bar \lambda,\bar x)\in \mathbb{R}\times (\mathbb{R}^n\backslash\{0\})$ be an $m$-degree $K$-eigenpair of $\mathcal{Q}$.  Since $\bar x\in K\backslash\{0\}$, by the definition of $K$, there exists a nonzero vector $\bar \alpha\in \mathbb{R}^p_+$ such that $\bar x=H^\top \bar \alpha$. Consequently, from the fact that $(\bar \lambda^m \mathcal{A}+\bar \lambda\mathcal{B}+\mathcal{C})\bar x^{m-1}\in K^*$ and the expression of $K^*$, it holds that $H(\bar \lambda^m \mathcal{A}+\bar \lambda\mathcal{B}+\mathcal{C})\bar x^{m-1}\geq 0$, which implies
\begin{equation}\label{EqC}
H(\bar \lambda^m \mathcal{A}+\bar \lambda\mathcal{B}+\mathcal{C})(H^\top \bar \alpha)^{m-1}\geq 0.
\end{equation}
By the definitions of $\mathcal{D}$, $\mathcal{G}$ and $\mathcal{S}$, we know that (\ref {EqC}) can be equivalently written as
$$
(\bar \lambda^m \mathcal{D} +\bar\lambda\mathcal{G}+\mathcal{S})\bar \alpha^{m-1}\geq 0.
$$
Moreover, it is easy to verify that $\langle \bar \alpha,(\bar \lambda^m \mathcal{D} +\bar\lambda\mathcal{G}+\mathcal{S})\bar \alpha^{m-1}\rangle=0$. Conversely, if $(\bar \lambda,\bar \alpha)\in \mathbb{R}\times (\mathbb{R}^p\backslash\{0\})$ satisfies (\ref{ECPEq}), then we can prove that $(\bar \lambda,\bar x)$ with $\bar x=H^\top\bar \alpha$ satisfies (\ref{QTECiP}) in a similar way.

Since $\mathcal{Q}$ is $K$-regular, it is easy to verify that $(\mathcal{D,G,S})$ is $\mathbb{R}^p_+$-regular. Consequently, by applying Theorem \ref{Proposition2} to problem (\ref{ECPEq}), we know that $\mathcal{Q}$ has at most $\btau_{m,p}=pm^{p}$ $m$-degree $K$-eigenvalues.
\qed\end{proof}

The above theorem shows that $\sigma(\mathcal{Q},K)$ has finitely many elements in case where $K$ is a polyhedral convex cone. However, the situation can be even worse in the nonpolyhedral case. For instance, Iusem and Seeger \cite{IS07} successfully constructed a symmetric matrix $C$ (i.e., $\mathcal{Q}=(O,I,-C)\in \mathcal{F}_{2,n}$) and a nonpolyhedral convex cone $K$ such that $\sigma(\mathcal{Q}, K)$ behaves like the Cantor ternary set, i.e., it is uncountable and totally disconnected.

\section{Optimization formulation of THDEiCP}\label{OptForm}

In this section, for the purpose of finding solutions of THDEiCP, we introduce an optimization reformulation, which paves the way of designing algorithms. Here, we only consider the case where $\mathcal{A}$ and $\mathcal{B}$ are two symmetric tensors, $\mathcal{C}:=-\mathcal{I}$, and $K:=\R^n_+$.

We consider the following homogeneous polynomial optimization problem.
\begin{equation}\label{EFOPt}
\begin{array}{cl}
{\rm max}&\varphi_0(u,v):=m(m-1)^{\frac{1}{m}-1}v^\top u^{[m-1]}-\mathcal{B}u^m\\
{\rm s.t.}&\mathcal{A}u^m+v^\top v^{[m-1]}=1,\\
&u\geq 0,v\geq0.
\end{array}
\end{equation}
Let $\phi_0(u,v):=\mathcal{A}u^m+v^\top v^{[m-1]}$, and let $\varphi_0$ be defined by \eqref{EFOPt}. We derive that
\begin{subnumcases}{\label{Delx}}
\nabla_u \varphi_0(u,v)=m(m-1)^{\frac{1}{m}}{\rm diag}(v )u^{[m-2]}-m\mathcal{B}u^{m-1},\label{Delxa}\\
\nabla_v \varphi_0(u,v)=m(m-1)^{\frac{1}{m}-1}u^{[m-1]},\label{Delxb}\\
\nabla_u\phi_0(u,v)=m\mathcal{A}u^{m-1},\label{Delxc} \\
\nabla_v\phi_0(u,v)=mv^{[m-1]}.\label{Delxd}
\end{subnumcases}

Now, we state the relationship between \eqref{QTECiP} and \eqref{EFOPt} as follows.
\begin{theorem}\label{EqEF}
Let $\mathcal{Q}=(\mathcal{A},\mathcal{B}, -\mathcal{I})\in \mathcal{F}_{m,n}$. Assume that $\mathcal{A}$ and $\mathcal{B}$ are both symmetric. Let $(\bar u,\bar v)$ with $\bar u\neq 0$ be a stationary point of problem \eqref{EFOPt}. Then $(\bar \lambda, \bar u)$ is an $m$-degree Pareto-eigenpair of $\mathcal{Q}$, where $\bar \lambda=(\varphi_0(\bar u,\bar v))^{\frac{1}{m-1}}$.
\end{theorem}

\begin{proof}
Since $(\bar u,\bar v)$ is a stationary point of problem (\ref{EFOPt}), it follows from \eqref{Delx} that there exist $\bar\alpha, \bar\beta \in \mathbb{R}^n$ and $\bar \gamma\in \mathbb{R}$, such that
\begin{subnumcases}{\label{Statpoint1}}
m\mathcal{B}\bar u^{m-1}-m(m-1)^{\frac{1}{m}}{\rm diag}(\bar v )\bar u^{[m-2]}=\bar \alpha+\bar \gamma m\mathcal{A}\bar u^{m-1}, \label{ST1}\\
-m(m-1)^{\frac{1}{m}-1}\bar u^{[m-1]}=\bar \beta+\bar \gamma m\bar v^{[m-1]},\label{ST2}\\
\bar \alpha\geq0,\;\bar u\geq0,\;\bar \alpha^\top \bar u=0,\label{ST3}\\
\bar \beta\geq0,\;\bar v\geq0,\;\bar \beta^\top \bar v=0,\label{ST4}\\
\mathcal{A}\bar u^m+\bar v^\top \bar v^{[m-1]}=1.\label{ST5}
\end{subnumcases}
Rearranging terms of \eqref{ST2} yields
\begin{equation}\label{bbuv}
-(m-1)^{\frac{1}{m}-1}\bar u^{[m-1]}-\bar \gamma \bar v^{[m-1]}=\bar\beta/m.
\end{equation}
We claim that $\bar \beta=0$. Otherwise, if $\bar \beta\neq0$, then there exists an index $i_0\in [n]$ such that $\bar \beta_{i_0}>0$, which implies $\bar v_{i_0}=0$ from \eqref{ST4}, and hence, it holds that $-(m-1)^{\frac{1}{m}-1}\bar u_{i_0}^{m-1}=\bar\beta_{i_0}/m>0$. It is a contradiction. Therefore, it follows from \eqref{bbuv} that
\begin{equation}\label{GGfd}
-(m-1)^{\frac{1}{m}-1}\bar u^{[m-1]}=\bar \gamma \bar v^{[m-1]}.
\end{equation}
Moreover, it is clear from the truth $\bar u\neq0$ and \eqref{GGfd} that $\bar \gamma<0$ and
\begin{equation}\label{GGfd1}
\bar v=(-\bar \gamma)^{-\frac{1}{m-1}}(m-1)^{-\frac{1}{m}} \bar u.
\end{equation}
By invoking \eqref{ST1} and \eqref{GGfd1}, we have
$$
\mathcal{B}\bar u^{m-1}-(-\bar \gamma)^{-\frac{1}{m-1}}\bar u^{[m-1]}-\bar \gamma \mathcal{A}\bar u^{m-1}=\frac{\bar\alpha}{m}\geq 0,\\
$$
which implies
\begin{equation}\label{Leq1}
(-\bar \gamma )^{\frac{m}{m-1}}\mathcal{A}\bar u^{m-1}+(-\bar \gamma)^{\frac{1}{m-1}}\mathcal{B}\bar u^{m-1}-\bar u^{[m-1]}\geq 0.
\end{equation}
Moreover, using \eqref{GGfd1}, \eqref{ST1}, and \eqref{ST3}, it is not difficult to verify that
\begin{equation}\label{Comp}
\langle \bar u, (-\bar \gamma )^{\frac{m}{m-1}}\mathcal{A}\bar u^{m-1}+(-\bar \gamma)^{\frac{1}{m-1}}\mathcal{B}\bar u^{m-1}-\bar u^{[m-1]}\rangle=0.
\end{equation}

On the other hand, it follows from \eqref{ST1} and \eqref{ST3} that
$$
m\mathcal{B}\bar u^m-m\bar \gamma \mathcal{A}\bar u^m-m(m-1)^{\frac{1}{m}}\bar v^\top\bar u^{[m-1]}=0,
$$
which implies
\begin{align}\label{ppgt1}
-m\varphi_0(\bar u,\bar v)&=m\bar \gamma \mathcal{A}\bar u^m-m(m-1)^{\frac{1}{m}-1}\bar v^\top\bar u^{[m-1]} \nn \\
&=m\bar \gamma \mathcal{A}\bar u^m+m\bar \gamma \bar v^\top\bar v^{[m-1]} \nn\\
&=m\bar \gamma,
\end{align}
where the second equality comes from \eqref{GGfd}, and the last equality is due to \eqref{ST5}. Hence, we conclude from \eqref{ppgt1} that $\varphi_0(\bar u,\bar v)=-\bar \gamma>0$, and both \eqref{Leq1} and \eqref{Comp} mean that $\bar u$ is an eigenvector of \eqref{EFOPt} associated to the eigenvalue $\bar \lambda$.
\qed\end{proof}

\begin{theorem}\label{QTqq12}
Let $\mathcal{Q}=(\mathcal{A},\mathcal{B}, -\mathcal{I})\in \mathcal{F}_{m,n}$. Assume that $\mathcal{A}$ and $\mathcal{B}$ are both symmetric and $\mathcal{A}$ is copositive. Let $(\bar\lambda, \bar x)$ be an $m$-degree Pareto-eigenpair of $\mathcal{Q}$. If $\bar \lambda>0$, then $(\bar u, \bar v)$ is a stationary point of \eqref{EFOPt}, where
\[\label{solu}(\bar u, \bar v)=\frac{1}{\left(\mathcal{A}\bar x^m+\bar y^\top \bar y^{[m-1]}\right)^\frac{1}{m}}(\bar x,\bar y)\]
with $\bar y={(m-1)^{-\frac{1}{m}}} ({\bar\lambda})^{-1}{\bar x}$.
\end{theorem}

\begin{proof}
Since $\bar y\in \mathbb{R}_+^n\backslash \{0\}$ and $\mathcal{A}$ is copositive, we have $\mathcal{A}\bar x^m+\bar y^\top \bar y^{[m-1]}>0$. Moreover, it is easy to check that $(\bar u, \bar v)$ given in \eqref{solu} is a feasible solution of (\ref{EFOPt}). Take $$\bar \alpha=\frac{m}{\bar\lambda(\mathcal{A}\bar x^m+\bar y^\top \bar y^{[m-1]})^{\frac{m-1}{m}}}\left(\bar \lambda^m\mathcal{A}\bar x^{m-1}+\bar \lambda \mathcal{B}\bar x^{m-1}-\bar x^{[m-1]}\right), $$
$\bar\beta=0$ and $\bar\gamma=-\bar \lambda^{m-1}.$ It is obvious that $\bar \alpha\geq 0$, since $(\bar \lambda, \bar x)$ satisfies (\ref{QTECiP}) and $\bar \lambda>0$. Moreover, it is not difficult to see that $\bar\alpha^\top \bar u=0$ and $\bar\beta^\top \bar v=0$. Finally, with the definition of $(\bar u,\bar v)$ given in \eqref{solu}, we can verify that \eqref{ST1} and \eqref{ST3} hold, which means that $(\bar u,\bar v)$ is a stationary point of (\ref{EFOPt}).
\qed
\end{proof}

Denote $w:=(u^\top,v^\top)^\top$ and $\phi_i(w)=w_i$ for $i=1,2,\ldots,2n$. When $\mathcal{A}$ is strictly copositive, the feasible set of \eqref{EFOPt} is compact. Hence, the globally optimal value of \eqref{EFOPt}, denoted by $\varphi_0^{\rm max}$, exists. Denote by $\bar w$ the corresponding globally optimal solution, and denote
$$\bar d=\left((e_{I(\bar w)}\right)^\top,-\bar t(\bar w_{I^c(\bar w)})^\top)^\top~~~{\rm  with}~~\bar t=\sum_{i\in I(\bar w)}\left(\mathcal{A}\bar u^{m-1}\right)_i,$$
 where $I(\bar w)=\{i\in [2n]~|~\bar w_i=0\}$ and $I^c(\bar w)=[n]\backslash I(\bar w)$. From the homogeneity of $\phi_0$ and the fact that $\phi_0(\bar w)=1$, it is easy to verify that $\bar w^\top \nabla \phi_0(\bar w)=m \phi_0(\bar w)=m\neq 0$, which implies that $\nabla \phi_0(\bar w)\neq 0$ and hence it is linearly independent. Moreover, it is not difficult to know that
$$
\bar d^\top \nabla \phi_0(\bar w)=m\left(\sum_{i\in I(\bar w)}(\mathcal{A}\bar u^{m-1})_i-\bar t\right)=0$$ and $\bar d^\top \nabla \phi_i(\bar w)=1>0$
for every $i\in I(\bar w)$. This means that the Mangasarian-Fromovitz constraint qualification (MFCQ) holds at $\bar w$. Therefore, we know that $\bar w$ is a stationary point of (\ref{EFOPt}). Moreover, we claim that $\bar u\neq 0$. In fact, by taking $u_t=te$ and $v_t=\left(\frac{1-\bar at^m}{n}\right)^{1/m}e$ with $\bar a=\sum_{i_1,\ldots,i_m=1}^na_{i_1\ldots i_m}$, we can see that $w_t=(u_t,v_t)$ is a feasible solution of \eqref{EFOPt}, and the corresponding objective value is $$\varphi_0(u_t,v_t)=t^m\left(nm(m-1)^{\frac{1}{m}-1}\left(\frac{1-\bar at^m}{n}\right)^{\frac{1}{m}}-\bar b t\right),$$
where $\bar b=\sum_{i_1,\ldots,i_m=1}^nb_{i_1\ldots i_m}$. Hence, we know that $\varphi_0(u_t,v_t)>0$ for $t>0$ enough small,  which implies that $\varphi_0(\bar u,\bar v)>0$ due to the fact that $(\bar u,\bar v)$ is an optimal
solution of problem (\ref{EFOPt}). Consequently, it holds that $\bar u\neq0$. Moreover, by Theorem \ref{EqEF}, we know that $(\bar \lambda, \bar u)$ with $\bar \lambda=(\varphi_0(\bar u,\bar v))^{\frac{1}{m-1}}$ is a solution of (\ref{QTECiP}), which implies that  (\ref{QTECiP}) has at least a positive $m$-degree Pareto-eigenvalue. Therefore, one has $\varphi_0^{\max}\leq \lambda_{\max}^{m-1}$,
where $$
\lambda_{\max}
={\max}\left\{\lambda\in \mathbb{R}~|~\exists~ x \in \mathbb{R}^n,~(\lambda,x)~{\rm is~ an~}m\text{-degree~Pareto-}\text{eigenpair of~}\mathcal{Q}\right\}.
%\end{array}
$$

\begin{theorem}\label{global2}
Let $\mathcal{Q}=(\mathcal{A},\mathcal{B},-\mathcal{I})\in \mathcal{F}_{m,n}$. Assume that $\mathcal{A}$ and $\mathcal{B}$ are both symmetric and $\mathcal{A}$ is strictly copositive. Then, we have
$$
\lambda_{\max}^{m-1}=\varphi_0^{\max}.
$$
\end{theorem}

\begin{proof}
Let $(\bar \lambda, \bar x)$ with $\bar \lambda>0$ be an $m$-degree Pareto-eigenpair of $\mathcal{Q}$. By the homogeneity of the complementarity system (\ref{QTECiP}) with respect to $x$, without loss of generality, we assume that $\bar x$ satisfies $e^\top \bar x=1$. As a consequence, it immediately follows from the strict copositiveness of $\A$ that $\mathcal{A}\bar x^m>0$. Denote
\begin{equation}\label{baruv}
\bar y=\frac{(m-1)^{-\frac{1}{m}}}{\bar \lambda}\bar x~~~~{\rm and}~~~~(\bar u,\bar v)=\frac{(\bar x,\bar y)}{(\mathcal{A}\bar x^m+\bar y^\top \bar y^{[m-1]})^{\frac{1}{m}}}.
\end{equation}
It is trivial that $(\bar u,\bar v)\in \mathbb{R}_+^n\times \mathbb{R}_+^n$,
$$\mathcal{A}\bar u^m=\frac{1}{\mathcal{A}\bar x^m+\bar y^\top \bar y^{[m-1]}}\mathcal{A}\bar x^m,~~~{\rm and}~~~\bar v^\top\bar v^{[m-1]}=\frac{1}{\mathcal{A}\bar x^m+\bar y^\top \bar y^{[m-1]}}\bar y^\top\bar y^{[m-1]},
$$
which implies that $\mathcal{A}\bar u^m+\bar v^\top \bar v^{[m-1]}=1$ holds. Hence, $(\bar u,\bar v)$ is a feasible solution of \eqref{EFOPt} and $\varphi_0(\bar u,\bar v)\leq \varphi_0^{\rm max}$.

 On the other hand, since $(\bar\lambda, \bar x)$ is an $m$-degree Pareto-eigenpair of $\mathcal{Q}$, we know that $\bar \lambda ^m\mathcal{A}\bar x^{m}+\bar \lambda\mathcal{B}\bar x^{m}-\bar x^\top \bar x^{[m-1]}=0$. Substituting $(\bar u,\bar v)$ into $\varphi_0(u,v)$ yields
\begin{align*}
\varphi_0(\bar u,\bar v)&=m(m-1)^{\frac{1}{m}-1}\bar v^\top \bar u^{[m-1]}-\mathcal{B}\bar u^{m}\\
&=\frac{m\bar x^\top \bar x^{[m-1]}-(m-1)\bar \lambda\mathcal{B}\bar x^m}{(m-1)\bar \lambda^m\mathcal{A}\bar x^m+\bar x^\top \bar x^{[m-1]}}\bar \lambda^{m-1}\\
&=\bar \lambda^{m-1},
\end{align*}
where the second equality comes from (\ref{baruv}). Therefore, it holds that $\bar\lambda^{m-1}\leq \varphi_0^{\rm max}$, which implies that $\lambda_{\rm max}^{m-1}\leq \varphi_0^{\rm max}$. To sum up, we obtain the desired result of this theorem.
\qed\end{proof}

\section{Existence of solutions for THDEiCP}\label{Existence}
In this section, we study more general results on the existence of solutions of THDEiCP with $\Cc:=\pm \I$ and $K=\mathbb{R}_+^n$, but without symmetry assumptions on $\A$ and $\Bc$. We first present the existence result of symmetric tensors.

\begin{theorem}\label{SymmExist}
Let $\mathcal{Q}=(\mathcal{A},\mathcal{B},-\mathcal{I})\in \mathcal{F}_{m,n}$. Assume that $\mathcal{A}$ and $\mathcal{B}$ are both symmetric and $\mathcal{A}$ is strictly copositive. Then $\mathcal{Q}$ has at least an $m$-degree Pareto-eigenpair.
\end{theorem}
\begin{proof}
Consider the homogeneous polynomial optimization problem (\ref{EFOPt}). Since $\mathcal{A}$ is strictly copositive, it is easy to see that the feasible set of (\ref{EFOPt}) is compact. Consequently, from the continuity of the objective function $\varphi_0$ in (\ref{EFOPt}), its globally optimal solution, denoted by $(\bar u, \bar v)$, exists.  As arguments above, the constraint qualification MFCQ holds at $(\bar u,\bar v)$. Hence, $(\bar u,\bar v)$ is a stationary point of (\ref{EFOPt}). Moreover, we know that $\varphi_0(\bar u,\bar v)>0$ from the arguments above, which implies that $\bar u\neq 0$. Therefore, the assertion of Theorem \ref{EqEF} shows that $(\bar \lambda, \bar u)$ is an $m$-degree Pareto-eigenpair of $\mathcal{Q}$, where $\bar \lambda=(\varphi_0(\bar u,\bar v))^{\frac{1}{m-1}}$.
\qed
\end{proof}

The above theorem is a fundamental result for THDEiCP. However, many real-world problems often violate the symmetry condition. In other words, the symmetry assumptions on $\A$ and $\Bc$ are relatively stronger. Indeed, a general existence theorem of solutions of QEiCP have been well established in \cite{S01}, which states that, {\it if $(A,B,C)$ satisfies co-hyperbolicity properties and the leading matrix $A$ is co-regular,
then the considered QEiCP has at least a solution}. As a generalization of QEiCP, we are naturally concerned with whether such a similar result of QEiCP also holds for tensors. Hereafter, we study such a more general result without assuming the symmetry of $\A$ and $\Bc$, in addition to presenting some checkable conditions on $\mathcal{Q}=(\mathcal{A,B,C})$ instead of the co-hyperbolicity.

\begin{theorem}\label{ExistTh}
Let $\mathcal{Q}=(\mathcal{A,B,-I})\in \mathcal{F}_{m,n}$. Assume that $\mathcal{A}$ is strictly copositive. If $\mathcal{A}$ and $\mathcal{B}$ satisfy the following condition
\[\label{condition}
(a_{ii\ldots i}+1-m)(m-1)^{\frac{1}{m}-1}-b_{ii\ldots i}>0, ~~~\forall~i\in [n],
\]
then $\mathcal{Q}$ has at least an $m$-degree Pareto-eigenpair.
\end{theorem}

\begin{proof}
We first denote two sets by
$$S=\{(x,y)\in \mathbb{R}^n_+\times \mathbb{R}^n_+~|~x\geq 0,e^\top x=1,y\geq0\}~~~{\rm and}~~~S_0=\{(x,y)\in S~|~\|y\|\leq 1\}.$$
It is clear that $S_0$ is a compact convex subset of $S$. Define the function $F:S\times S\rightarrow \mathbb{R}$ by
\begin{align}\label{Fun}
F(x,y;z,w)=&\langle - \mathcal{B}x^{m-1}-f(x,y)\mathcal{A}x^{m-1}+(m-1)^{\frac{1}{m}}{\rm diag}(y)x^{[m-2]},z\rangle \nn\\
&+\langle (m-1)^{\frac{1}{m}-1}x^{[m-1]}-f(x,y)y^{[m-1]},w\rangle,
\end{align}
where
$$
f(x,y)=\frac{m(m-1)^{\frac{1}{m}-1}y^\top x^{[m-1]}-\mathcal{B}x^m}{\mathcal{A}x^m+y^\top y^{[m-1]}}.
$$
Apparently, $F(x,y;x,y)=0$ holds for any $(x,y)\in S$. Moreover, it can be seen that $F(\cdot,\cdot;z,w)$ is lower-semicontinuous on $S$ for any fixed $(z,w)\in S$, and $F(x,y;\cdot,\cdot)$ is concave on $S$ for any fixed $(x,y)\in S$. With the given condition \eqref{condition}, we claim that
$$\Omega:=\{(z,w)\in S~|~F(x,y;z,w)\leq 0,~\forall~(x,y)\in S_0\}$$
is compact. Otherwise, if $\Omega$ is not compact, then exists a sequence $\{(z^{(k)},w^{(k)})\}$ of $\Omega$ such that
$$\|(z^{(k)},w^{(k)})\|\rightarrow +\infty\quad \text{as}\quad k\rightarrow +\infty.$$
Since $\{z^{(k)}\}$ is bounded, without loss of generality, we claim that $\|w^{(k)}\|\rightarrow +\infty$. As a consequence, there exists $i_0\in [n]$ such that $w_{i_0}^{(k)}\rightarrow +\infty$. By taking $x^{(k)}=y^{(k)}=e_{i_0}\in S_0$ with $e_{i_0}$ being the $i_0$-th unit vector in $\mathbb{R}^n$, we have
$$
F(x^{(k)},y^{(k)};z^{(k)},w^{(k)})=\theta_k+\frac{(a_{i_0\ldots i_0}+1-m)(m-1)^{\frac{1}{m}-1}-b_{i_0\ldots i_0}}{a_{i_0\ldots i_0}+1}w_{i_0}^{(k)},
$$
where
$$\theta_k=\langle - \mathcal{B}(x^{(k)})^{m-1}-f(x^{(k)},y^{(k)})\mathcal{A}(x^{(k)})^{m-1}+(m-1)^{\frac{1}{m}}{\rm diag}(y^{(k)})(x^{(k)})^{[m-2]},z^{(k)}\rangle.$$
Clearly, the sequence $\{\theta_k\}$ is bounded. It follows from the condition \eqref{Concl} that $$F(x^{(k)},y^{(k)};z^{(k)},w^{(k)})>0$$
for enough large $k$, which contradicts the fact that $(z^{(k)},w^{(k)})\in\Omega$. By Theorem 6 in \cite{KF84}, there exists $(\bar x,\bar y)\in S$ such that
\begin{equation}\label{KFIn}
F(\bar x,\bar y;z,w)\leq 0,~~~\forall~(z,w)\in S.
\end{equation}
Take $w=0$ in \eqref{KFIn}, we know that, for any $z\in D:=\{z\in \mathbb{R}^n~|~z\geq 0,e^\top z=1\}$,
$$
F(\bar x,\bar y;z,0)=\langle - \mathcal{B}\bar x^{m-1}-\bar f\mathcal{A}\bar x^{m-1}+(m-1)^{\frac{1}{m}}{\rm diag}(\bar y)\bar x^{[m-2]},z\rangle\leq 0,
$$
where $\bar f=f(\bar x,\bar y)$, which implies
\begin{equation}\label{ABC}
\mathcal{B}\bar x^{m-1}+\bar f\mathcal{A}\bar x^{m-1}-(m-1)^{\frac{1}{m}}{\rm diag}(\bar y)\bar x^{[m-2]}\geq 0,
\end{equation} since $D$ is a basis of $\mathbb{R}_+^n$. Take again any $w\in \mathbb{R}_+^n$, it is clear that $(\bar x,\bar y+w)\in S$. Consequently, it holds that
\begin{align*}
F(\bar x,\bar y;\bar x, \bar y+w)&=F(\bar x,\bar y;\bar x, \bar y)+\langle (m-1)^{\frac{1}{m}-1}\bar x^{[m-1]}-\bar f\bar y^{[m-1]}, w\rangle\\
&=\langle (m-1)^{\frac{1}{m}-1}\bar x^{[m-1]}-\bar f\bar y^{[m-1]}, w\rangle \\
&\leq 0,
\end{align*}
where the second equality is due to the fact that $F(\bar x,\bar y;\bar x, \bar y)=0$. Hence,
\begin{equation}\label{Axy}
\bar f\bar y^{[m-1]}-(m-1)^{\frac{1}{m}-1}\bar x^{[m-1]}\geq 0.
\end{equation}
Since $\bar x\geq 0$ and $\bar x\neq 0$, there exists $i_0\in [n]$ such that $\bar x_{i_0}> 0$. Accordingly,
$$
\bar f\bar y_{i_0}^{m-1}\geq (m-1)^{\frac{1}{m}-1}\bar x_{i_0}^{m-1}>0,
$$
which implies $\bar f>0$. Denote $I(\bar y)=\{i\in [n]~|~\bar y_i=0\}$. It is clear that $I(\bar y)$ is a proper subset of $[n]$. By (\ref{Axy}), it is obvious that $\bar x_i=0$ for any $i\in I(\bar y)$. So
\begin{equation}\label{Iy}
\bar f\bar y_i^{m-1}=(m-1)^{\frac{1}{m}-1}\bar x_i^{m-1},~~~~\forall ~i\in I(\bar y).
\end{equation}
For any $i\in [n]\backslash I(\bar y)$, taking $w=te_i$ with $t\in \mathbb{R}$, it follows from $\bar y_i>0$ that $(\bar x,\bar y+w)\in S$ for any real number $t$ with enough small $|t|$. Recalling \eqref{KFIn}, we have
$$
F(\bar x,\bar y;\bar x, \bar y+w)\leq 0,
$$
which implies that
$$ \left((m-1)^{\frac{1}{m}-1}\bar x_i^{m-1}-\bar f\bar y_i^{m-1}\right)t\leq 0 $$
for $i\in [n]\backslash I(\bar y)$ and any real number $t$ with enough small $|t|$. We immediately obtain
\begin{equation}\label{Icy}
(m-1)^{\frac{1}{m}-1}\bar x_i^{m-1}=\bar f\bar y_i^{m-1},~~~\forall~i\in [n]\backslash I(\bar y).
 \end{equation}
By \eqref{Iy} and \eqref{Icy}, it holds that
\begin{equation}\label{Eqy}
\bar f\bar y^{[m-1]}=(m-1)^{\frac{1}{m}-1}\bar x^{[m-1]},
 \end{equation}
or equivalently,
\begin{equation}\label{Eqy2}
\bar f^{\frac{1}{m-1}}\bar y=(m-1)^{-\frac{1}{m}}\bar x.
 \end{equation}
Combining \eqref{ABC} and \eqref{Eqy} leads to
\begin{align*}
0&\leq\mathcal{B}\bar x^{m-1}+\bar f\mathcal{A}\bar x^{m-1}-\bar f^{-\frac{1}{m-1}}\bar x^{[m-1]}\\
&=\bar f^{-\frac{1}{m-1}}\left\{\bar f^{\frac{1}{m-1}}\mathcal{B}\bar x^{m-1}+\bar f^{\frac{m}{m-1}}\mathcal{A}\bar x^{m-1}-\bar x^{[m-1]}\right\},
\end{align*}
which implies
\begin{equation}
\bar \lambda ^m\mathcal{A}\bar x^{m-1}+\bar \lambda\mathcal{B}\bar x^{m-1}-\bar x^{[m-1]}\geq 0,
\end{equation}
where $\bar \lambda=\bar f^{\frac{1}{m-1}}$. Now we verify that
$$\left\langle \bar x,\; \bar \lambda ^m\mathcal{A}\bar x^{m-1}+\bar \lambda\mathcal{B}\bar x^{m-1}-\bar x^{[m-1]}\right\rangle=0.$$
We only need to verify
$$\left\langle \bar x, \bar f\mathcal{A}\bar x^{m-1}+\mathcal{B}\bar x^{m-1}-\bar f^{-\frac{1}{m-1}}\bar x^{[m-1]}\right\rangle=0,$$
that is,
$$\bar f\mathcal{A}\bar x^{m}+\mathcal{B}\bar x^{m}-\bar f^{-\frac{1}{m-1}}\sum_{i=1}^n\bar x_i^m=0.$$
Since $F(\bar x,\bar y;\bar x,\bar y)=0$, that is,
$$\bar f\mathcal{A}\bar x^{m}+\mathcal{B}\bar x^{m}=m(m-1)^{\frac{1}{m}-1}\bar y^\top \bar x^{[m-1]}-\bar f \sum_{i=1}^n\bar y_i^m,$$
we only need to further verify
\begin{equation}\label{mmn}
m(m-1)^{\frac{1}{m}-1}\bar y^\top \bar x^{[m-1]}-\bar f \sum_{i=1}^n\bar y_i^m-\bar f^{-\frac{1}{m-1}}\sum_{i=1}^n\bar x_i^m=0.
\end{equation}
Actually, the left hand of \eqref{mmn} amounts to
\begin{align*}
&\;m(m-1)^{\frac{1}{m}-1}\bar y^\top \bar x^{[m-1]}-\bar f \bar y^\top \bar y^{[m-1]}-\bar f^{-\frac{1}{m-1}}\sum_{i=1}^n\bar x_i^m\\
&\; =m(m-1)^{\frac{1}{m}-1}\bar y^\top \bar x^{[m-1]}-(m-1)^{\frac{1}{m}-1}\bar y^\top \bar x^{[m-1]}-\bar f^{-\frac{1}{m-1}}\sum_{i=1}^n\bar x_i^m\\
&\; =(m-1)^{\frac{1}{m}}\bar y^\top \bar x^{[m-1]}-\bar f^{-\frac{1}{m-1}}\bar x^\top \bar x^{[m-1]}\\
&\; =(m-1)^{\frac{1}{m}}\bar y^\top \bar x^{[m-1]}-\bar f^{-\frac{1}{m-1}}(m-1)^{\frac{1}{m}}\bar f^{\frac{1}{m-1}}\bar y^\top \bar x^{[m-1]}\\
&\;=0,
\end{align*}
where the first equality is due to (\ref{Eqy}), and the second equality comes from (\ref{Eqy2}).
Therefore, $(\bar \lambda, \bar x)$ is an $m$-degree Pareto-eigenpair of $\mathcal{Q}$.
\qed
\end{proof}

Similarly, when we deal with the case of $\Cc:=\I$, we can also establish the following result.
 \begin{theorem}\label{ExistTh2}
Let $\mathcal{Q}=(\mathcal{A,B,I})\in \mathcal{F}_{m,n}$. Assume that $\mathcal{A}$ is strictly copositive. If $\mathcal{A}$ and $\Bc$ satisfy the following condition
$$
(m+a_{ii\ldots i}-1)(m-1)^{\frac{1}{m}-1}+b_{ii\ldots i}>0, ~~~\forall~i=1,2,\ldots,n,
$$
then $\mathcal{Q}$ has at least an $m$-degree Pareto-eigenpair.
\end{theorem}
\begin{proof}
Define the function $h:\mathbb{R}_+^n\times \mathbb{R}_+^n\rightarrow \mathbb{R}$ by
$$%\begin{equation}\label{hun}
h(x,y)=\frac{(2-m)(m-1)^{\frac{1}{m}-1}y^\top x^{[m-1]}-\mathcal{B}x^m}{\mathcal{A}x^m+y^\top y^{[m-1]}}.
$$%\end{equation}
and the function $G:S\times S\rightarrow \mathbb{R}$ by
$$
\begin{array}{lll}
G(x,y;z,w)&=&\langle - \mathcal{B}x^{m-1}-h(x,y)\mathcal{A}x^{m-1}-(m-1)^{\frac{1}{m}}{\rm diag}(y)x^{[m-2]},z\rangle\\
&&+\langle (m-1)^{\frac{1}{m}-1}x^{[m-1]}-h(x,y)y^{[m-1]},w\rangle,
\end{array}
$$
where $S$ is defined in the proof of Theorem \ref{ExistTh}. We can prove the assertion in a similar way that used in Theorem \ref{ExistTh}, and skip its details here. \qed
\end{proof}

As a byproduct of Theorem \ref{ExistTh2}, we immediately  obtain the following existence result of the solution for QEiCP, which differs from the one presented in \cite{S01}.
\begin{corollary}
Consider QEiCP corresponding to the special case of THDEiCP with $m=2$. Let ${\mathcal Q}:=(A,B,I)\in \mathcal{M}_n$. Assume that $A$ is strictly copositive matrix. If $A$ and $B$ satisfy that $a_{ii}+b_{ii}+1>0$ for every $i\in [n]$, then ${\mathcal Q}$ has at least one quadratic Pareto-eigenpair.
\end{corollary}

\section{Numerical algorithm and experiments}\label{NumAlg}
In this section, we first introduce an implementable splitting algorithm based upon the augmented Lagrangian method, which efficiently exploits the weakly coupled structure of the resulting optimization formulation of THDEiCP. Then, we conduct some computational results to show the reliability and convergence behavior of the proposed algorithm.

\subsection{The algorithm}\label{SecAlg}

Note that model \eqref{EFOPt} can be recast as the standard minimization problem:
\begin{align}
\min &\;\; \Bc u^m + {\btheta}v^\t u^{[m-1]}\nn \\
{\rm s.t.} &\;\; \A u^m + \I v^m = 1, \label{mini}\\
&\;\; u\geq 0,\;\;v\geq 0,\nn
\end{align}
where $\btheta$ is a constant given by $\btheta:=-m(m-1)^{\frac{1}{m}-1}$. Here, we should notice that it is possible to employ the powerful {\it semismooth} and {\it smoothing} Newton methods to solve the model under consideration. However, we show below that a first-order structure-exploiting algorithm can be developed, which is much easier to be implemented than the second-order type methods.

Taking a revisit on \eqref{mini}, we observe that \eqref{mini} is an equality constrained optimization problem, and we know that the Augmented Lagrangian Method (ALM) \cite{Hes69,Pow69} is a benchmark solver for this type model. Let $\bzeta\in\R$ be the Lagrangian multiplier associated to the equality constraint. The augmented Lagrangian function is given by
\[\label{ALF} \L(u,v,\bzeta):= \Bc u^m + \btheta v^\t u^{[m-1]} -\bzeta\left(\A u^m + \I v^m -1\right) + \frac{\beta}{2}\left(\A u^m + \I v^m -1\right)^2,\]
where $\beta>0$ is the penalty parameter. Consequently, for a given $\bzeta^{(k)}\in\R$, the iterative scheme of ALM reads as follows:
\begin{subnumcases}{\label{ALM}}
(u^{(k+1)},v^{(k+1)})=\arg\min_{u,v} \left\{\L(u,v,\bzeta^{(k)})\;|\;u\geq 0,\;v\geq0\right\}; \label{ALMa}\\
\bzeta^{(k+1)}=\bzeta^{(k)} - \beta \lb \A (u^{(k+1)})^m +\I (v^{(k+1)})^m -1\rb. \label{ALMb}
\end{subnumcases}
However, it seems not easy enough to implement such an algorithm due to the {\it coupled structure} and {\it high nonlinearity} emerging in the objective function and equality constraint. To improve its implementability and numerical performance, the so-called Alternating Direction Method of Multipliers (ADMM) \cite{GM76,GM75} was judiciously developed for {\it separable} convex minimizations by updating the variables in an {\it alternating} (Gauss-Seidel) order. In recent years, it is well documented that ADMM has a surge of popularity in the areas such as signal/image processing, statistical learning, data mining, and so on. Here we just refer to \cite{BPCPE10,EY15,Glo14} for some surveys on ADMM.

Following the spirit of ADMM, we split the first subproblem \eqref{ALMa} into two parts. For given $(v^{(k)},\bzeta^{(k)})$, we immediately have the following ADMM scheme:
\begin{subnumcases}{\label{ADM}}
u^{(k+1)}=\arg\min_{u} \left\{\L(u,v^{(k)},\bzeta^{(k)})\;|\;u\geq 0\right\}; \label{ADMa}\\
v^{(k+1)}=\arg\min_{v} \left\{\L(u^{(k+1)},v,\bzeta^{(k)})\;|\;v\geq0\right\}; \label{ADMb}\\
\bzeta^{(k+1)}=\bzeta^{(k)} - \beta \lb \A (u^{(k+1)})^m +\I (v^{(k+1)})^m -1\rb. \label{ADMc}
\end{subnumcases}
It seems that such an algorithm exploits the weakly separable structure of model \eqref{mini}. However, it also fails to be easily implemented, because the first two subproblems are not easy enough to have closed-form solutions. Indeed, we can clearly observe that both subproblems \eqref{ADMa} and \eqref{ADMb} have very simple convex sets as their constraints, thereby making the projections onto these sets very easy. Hence, it would greatly simplify \eqref{ADM} if both subproblems could reduce to the computation of projections. Below, we consider the linearized version of \eqref{ADM} so that each subproblem has closed-form representation. Since $\L(u,v,\bzeta)$ is nonconvex with respect to $u$ and $v$ in general cases, for the purpose of making both subproblems well-posed, we attach two proximal terms $\frac{\gamma_1}{2}\|u-u^{(k)}\|^2$ and $\frac{\gamma_2}{2}\|v-v^{(k)}\|^2$ to \eqref{ADMa} and \eqref{ADMb}, respectively. Here $\gamma_1$ and $\gamma_2$ are two positive constants. More specifically, linearizing the nonlinear parts of $\L(u,v,\bzeta)$ (see gradients in \eqref{Delx}), we derive a linearized ADMM as follows:
\begin{subnumcases}{\label{LADM}}
u^{(k+1)} = \Pi_{\R^n_{+}}\left[u^{(k)} -\frac{\bPhi^{(k)} }{\gamma_1}  \right], \label{LADM1} \\
v^{(k+1)} = \Pi_{\R^n_{+}}\left[v^{(k)}- \frac{\btheta (u^{(k+1)})^{[m-1]}+\bUpsilon^{(k)}}{\gamma_2}  \right], \label{LADM2}\\
\bzeta^{(k+1)} = \bzeta^{(k)} - \beta\left(\A (u^{(k+1)})^m + \I (v^{(k+1)})^m -1\right), \label{LADM3}
\end{subnumcases}
where $\Pi_{\R^n_+}[\cdot]$ represents the projection onto $\R^n_+$;
$$ \bPhi^{(k)}:=m\Bc (u^{(k)})^{m-1} +\btheta (m-1){\rm diag}(v^{(k)})(u^{(k)})^{[m-2]} +\beta m {\bm q}^{(k)}\A (u^{(k)})^{m-1},$$
with ${\bm q}^{(k)} := \A (u^{(k)})^m+\I (v^{(k)})^m-1-\frac{\bzeta^{(k)}}{\beta}$; and
$$ \bUpsilon^{(k)}:=\beta m\left(\A (u^{(k+1)})^m +\I (v^{(k)})^m-1-\frac{\bzeta^{(k)}}{\beta} \right)\I (v^{(k)})^{m-1}.$$
Obviously, the linearized version \eqref{LADM} is more implementable than \eqref{ADM} due to the pretty simple iterative scheme. To the best of our knowledge, there is no convergence result of such a linearized ADMM for solving the underlying {\it nonconvex} model. Therefore, it seems that our method \eqref{LADM} goes beyond the theoretical guarantees of the traditional ADMM. However, we will illustrate that our method \eqref{LADM} indeed is numerically convergent for model \eqref{mini} in many cases.

\subsection{Numerical experiments}
We have shown theoretically that THDEiCP \eqref{QTECiP} is solvable when $K:=\R^n_+$ in Section \ref{Existence} and introduce an implementable splitting method in Section \ref{SecAlg}. Now, we turn our attention to verifying our theoretical results and convergence behavior of the proposed algorithm \eqref{LADM} through preliminary computational results. We implement our algorithm by {\sc Matlab} R2012b and conduct the numerical experiments on a Lenovo notebook with Intel(R) Core(TM) i5-5200U CPU@2.20GHz and 4GB RAM running on Windows 7 Home Premium operating system.

Notice that model \eqref{mini} is also available for matrix cases, and it is a new formulation in the QEiCP literature. Thus, we here also test a matrix scenario for the purpose of showing the efficiency of our proposed algorithm \eqref{LADM} in solving QEiCP. In the following experiments, we test three synthetic examples and only list the details of $\A$ (or $A$) and $\Bc$ (or $B$) in the coming examples. For the random data, we generate them by the {\sc Matlab} script `\verb"rand"'.

\begin{example}\label{exam1}
This example considers a special case of THDEiCP, that is, QEiCP, whose matrices $A$ and $B$ are uniformly distributed in $(0,1)$ and given by
$$A=\left(\begin{array}{cccc}
    0.2296  &  0.6870  &  0.7421  &  0.8943 \\
    0.6870  &  0.9403  &  0.1194  &  0.5919 \\
    0.7421  &  0.1194  &  0.9325  &  0.7779 \\
    0.8943  &  0.5919  &  0.7779  &  0.3290
\end{array}\right),\;\;B=\left(\begin{array}{cccc}
    0.2235  &  0.3014 &   0.7879  &  0.5394\\
    0.3014  &  0.4026 &   0.5329  &  0.5453\\
    0.7879  &  0.5329 &   0.8272  &  0.5375\\
    0.5394  &  0.5453 &   0.5375  &  0.5994
\end{array}\right).$$
\end{example}

\begin{example}\label{exam2}
We consider the case where $\A$ and $\Bc$ are two 3-rd order 4-dimensional tensors; $\mathcal{A}$ is strictly copositive, but not nonnegative, and $\Bc$ is a randomly generated tensor, whose entries are uniformly distributed in $(1,2)$, that is,
$$
\mathcal{A}(:,:,1)=
\left(\begin{array}{cccc}
2&2&4/3&4/3\\
2&4/3&2/3&4/3\\
4/3&2/3&8/3&0\\
4/3&4/3&0&2
\end{array}\right),~~
\mathcal{B}(:,:,1)=
\left(\begin{array}{cccc}
    1.6557 &   1.3572 &   1.7523  &  1.6055\\
    1.3572 &   1.7577 &   1.4572  &  1.2192\\
    1.7523 &   1.4572 &   1.7060  &  1.0645\\
    1.6055 &   1.2192 &   1.0645  &  1.8235
\end{array}\right),
$$
$$
\mathcal{A}(:,:,2)=
\left(\begin{array}{cccc}
2&4/3&2/3&4/3\\
4/3&12&-2/3&10/3\\
2/3&-2/3&16/3&-2\\
4/3&10/3&-2&14/3
\end{array}\right),~~
\mathcal{B}(:,:,2)=
\left(\begin{array}{cccc}
    1.6551 &   1.5612 &   1.4351 &   1.6946\\
    1.5612 &   1.3404 &   1.4202 &   1.5916\\
    1.4351 &   1.4202 &   1.5060 &   1.6231\\
    1.6946 &   1.5916 &   1.6231 &   1.1386
\end{array}\right),
$$
$$
\mathcal{A}(:,:,3)=
\left(\begin{array}{cccc}
4/3&2/3&8/3&0\\
2/3&-2/3&16/3&-2\\
8/3&16/3&6&2/3\\
0&-2&2/3&0
\end{array}\right)
,~~
\mathcal{B}(:,:,3)=
\left(\begin{array}{cccc}
    1.9172 &   1.3331 &   1.6440 &   1.6613\\
    1.3331 &   1.5678 &   1.4275 &   1.2617\\
    1.6440 &   1.4275 &   1.9340 &   1.0709\\
    1.6613 &   1.2617 &   1.0709 &   1.3371
\end{array}\right),
$$
$$
\mathcal{A}(:,:,4)=
\left(
\begin{array}{cccc}
4/3&4/3&0&2\\
4/3&10/3&-2&14/3\\
0&-2&2/3&0\\
2&14/3&0&4
\end{array}
\right),~~
\mathcal{B}(:,:,4)=
\left(\begin{array}{cccc}
    1.5383  &  1.5514 &   1.4477  &  1.3513\\
    1.5514  &  1.9619 &   1.4367  &  1.7875\\
    1.4477  &  1.4367 &   1.0844  &  1.4156\\
    1.3513  &  1.7875 &   1.4156  &  1.9106
\end{array}\right).
$$
\end{example}

\begin{example}\label{exam3}
We consider two $4$-th order $3$-dimensional symmetric tensors $\A$ and $\Bc$, where $\A$ and $\Bc$ are randomly generated and uniformly distributed in $(0,1)$. Specifically, $\A$ and $\Bc$ are taken as follows:
$$\A(:,:,1,1)=\lb\begin{array}{ccccc}
  0.6229  &&  0.2644  &&  0.3567 \\
    0.2644  &&  0.0475  && 0.7367\\
    0.3567  &&  0.7367 &&   0.1259
\end{array}\rb,\; \Bc(:,:,1,1)=\lb\begin{array}{ccccc}
    0.6954  &&  0.4018  &&  0.1406\\
    0.4018  &&  0.9957  &&  0.0483\\
    0.1406  &&  0.0483  &&  0.0988
\end{array}\rb, $$
$$\A(:,:,1,2)=\lb\begin{array}{ccccc}
  0.7563  &&  0.5878  &&  0.5406\\
    0.5878  &&  0.1379  &&  0.0715\\
    0.5406  &&  0.0715  &&  0.3725
\end{array}\rb, \;\Bc(:,:,1,2)=\lb\begin{array}{ccccc}
    0.6730 &&   0.5351 &&   0.4473\\
    0.5351 &&   0.2853 &&   0.3071\\
    0.4473 &&   0.3071 &&   0.9665
\end{array}\rb,$$
$$\A(:,:,1,3)=\lb\begin{array}{ccccc}
    0.0657  &&  0.4918  &&  0.9312\\
    0.4918  &&  0.7788  &&  0.9045\\
    0.9312  &&  0.9045  &&  0.8711
\end{array}\rb,\; \Bc(:,:,1,3)=\lb\begin{array}{ccccc}
    0.7585 &&   0.6433 &&  0.2306\\
    0.6433 &&   0.8986 &&   0.3427\\
    0.2306 &&   0.3427 &&   0.5390
\end{array}\rb, $$
$$\A(:,:,2,1)=\lb\begin{array}{ccccc}
    0.7563  && 0.5878  &&  0.5406\\
    0.5878  &&  0.1379 &&  0.0715\\
    0.5406  &&  0.0715 &&  0.3725
\end{array}\rb,\Bc(:,:,2,1)=\lb\begin{array}{ccccc}
    0.6730  &&  0.5351 &&   0.4473\\
    0.5351  &&  0.2853 &&   0.3071\\
    0.4473  &&  0.3071 &&   0.9665
\end{array}\rb, $$
$$\A(:,:,2,2)=\lb\begin{array}{ccccc}
    0.7689  &&  0.3941  &&  0.6034\\
    0.3941  &&  0.3577  &&  0.3465\\
    0.6034  &&  0.3465  &&  0.4516
\end{array}\rb,\;\Bc(:,:,2,2)=\lb\begin{array}{ccccc}
    0.3608  &&   0.3914 &&   0.5230\\
    0.3914  &&  0.6822  &&  0.5516\\
    0.5230  &&  0.5516  &&  0.7091
\end{array}\rb,$$
$$ \A(:,:,2,3)=\lb\begin{array}{ccccc}
    0.8077 &&   0.4910  &&  0.2953\\
    0.4910 &&   0.5054  &&  0.5556\\
    0.2953 &&   0.5556  &&  0.9608
\end{array}\rb,\; \Bc(:,:,2,3)=\lb\begin{array}{ccccc}
    0.4632 &&   0.2043 &&   0.2823\\
    0.2043 &&   0.7282 &&   0.7400\\
    0.2823 &&   0.7400 &&   0.9369
\end{array}\rb,$$
$$\A(:,:,3,1)=\lb\begin{array}{ccccc}
     0.0657  &&  0.4918  &&  0.9312\\
    0.4918  &&  0.7788 &&   0.9045\\
    0.9312  &&  0.9045 &&   0.8711
\end{array}\rb,\; \Bc(:,:,3,1)=\lb\begin{array}{ccccc}
    0.7585 &&   0.6433 &&   0.2306\\
    0.6433 &&   0.8986 &&   0.3427\\
    0.2306 &&   0.3427 &&   0.5390
\end{array}\rb, $$
$$\A(:,:,3,2)=\lb\begin{array}{ccccc}
 0.8077  &&  0.4910  &&  0.2953\\
    0.4910  &&  0.5054  &&  0.5556\\
    0.2953  &&  0.5556  &&  0.9608
\end{array}\rb, \;\Bc(:,:,3,2)=\lb\begin{array}{ccccc}
    0.4632  &&  0.2043 &&   0.2823\\
    0.2043  &&  0.7282 &&   0.7400\\
    0.2823  &&  0.7400 &&   0.9369
\end{array}\rb,$$
$$\A(:,:,3,3)=\lb\begin{array}{ccccc}
    0.7581 &&   0.7205 &&   0.9044\\
    0.7205 &&   0.0782 &&   0.7240\\
    0.9044 &&   0.7240 &&   0.3492
\end{array}\rb,\;\Bc(:,:,3,3)=\lb\begin{array}{ccccc}
    0.8200  &&  0.5914  &&  0.4983\\
    0.5914  &&  0.0762  &&  0.2854\\
    0.4983  &&  0.2854  &&  0.1266
\end{array}\rb.$$
\end{example}

Before our experiments, we first introduce a reasonable stopping rule for the proposed method \eqref{LADM}. Without loss of generality, we can use
\[\label{stopping} {\rm RelErr}:=\max\left\{\|u^{(k+1)}-u^{(k)}\|,\;\|v^{(k+1)}-v^{(k)}\|,\;|{\bm V}^{(k)}|\right\}\leq {\rm Tol}\]
as a termination criterion to pursue an approximate solution with a preset tolerance `Tol', where $|{\bm V}^{(k)}|:=|\A (u^{(k+1)})^m+\I (v^{(k+1)})^m-1|$ measures the violation of the underlying equality constraint. For the parameters involved in our algorithm, we throughout take $\beta=1$, in addition to setting the starting points $u^{(0)}$ and $v^{(0)}$ as randomly generated vectors and $\bzeta^{(0)}=0$. For the other two parameters, we choose $\gamma_1=200$ and $\gamma_2=10$ for Example \ref{exam1}, $\gamma_1=1000$ and $\gamma_2=50$ for Examples \ref{exam2}--\ref{exam3}. The tolerance `Tol' in \eqref{stopping} is taken as ${\rm Tol}=10^{-6}$ for all tests.

As we have mentioned in introduction, tensor-related polynomial optimization problems suffers from {\it high nonlinearity}. From a theoretical point of view, {\it linearization} may destroy structural properties of the underlying functions, thereby resulting in nonideal approximations so that the algorithm is not necessarily convergent for some cases. To investigate the performance of such a linearization, in Fig. \ref{fig1}, we plot  evolutions of the relative error (`RelErr' defined by \eqref{stopping}) and the objective value (`Obj.') of \eqref{mini} [i.e., $-\varphi_0(u^{(k)},v^{(k)})$] with respect to the number of iterations, respectively, and the ability of finding ideal solutions of THDEiCP.

\begin{figure}[!htbp]
\centering
\includegraphics[width=0.49\textwidth]{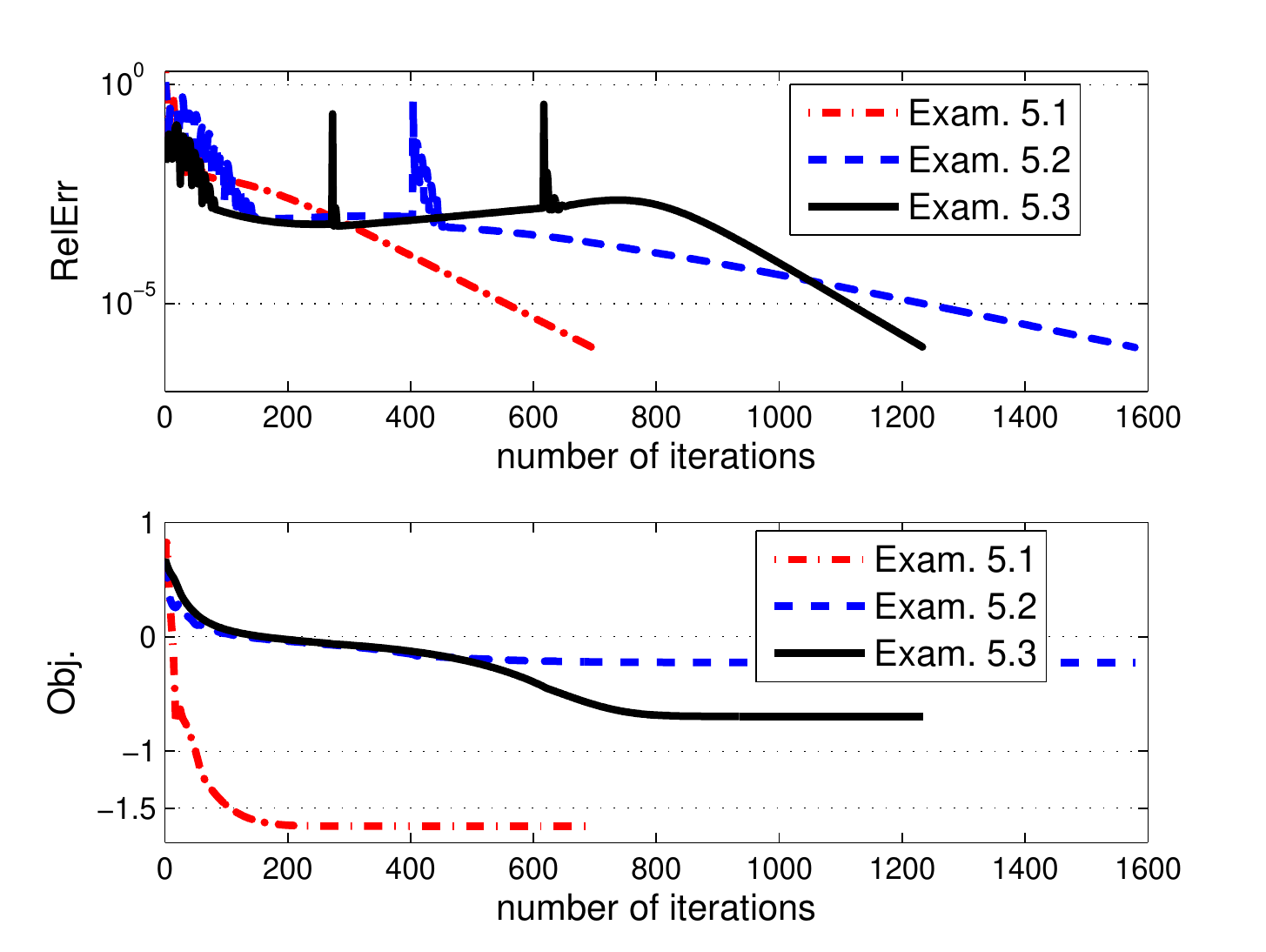}
\includegraphics[width=0.49\textwidth]{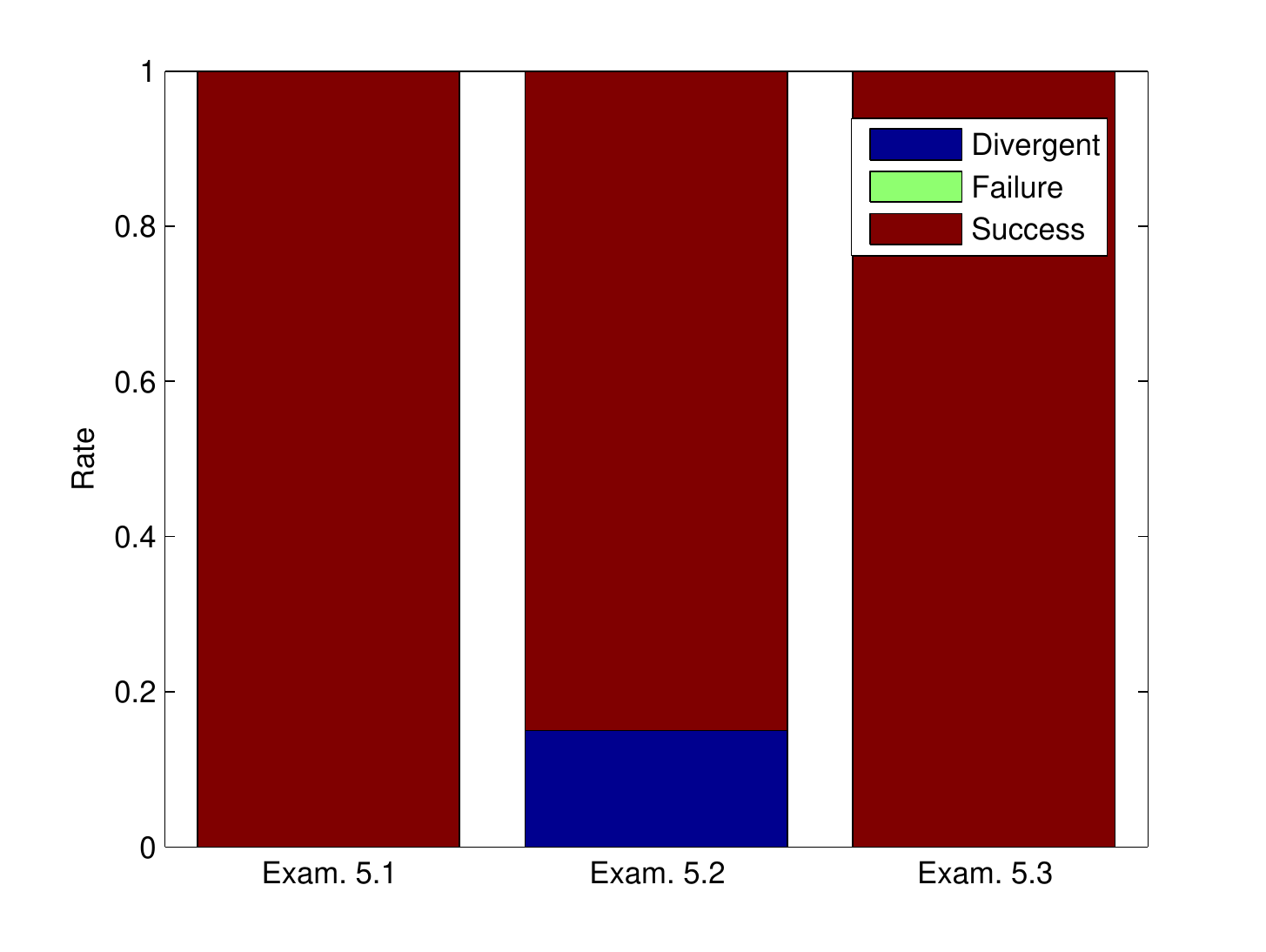}
\caption{Performance of the proposed algorithm. The left plot corresponds to the evolutions of `RelErr' defined by \eqref{stopping} and the objective value of \eqref{mini} [i.e., $-\varphi_0(u^{(k)},v^{(k)})$] with respect to the number of iterations, respectively. The right one shows the ability and reliability of the algorithm by testing $100$ randomly generated starting points.}\label{fig1}
\end{figure}

From the left plots of Fig. \ref{fig1}, we can see that our linearized ADMM is convergent very fast with a random starting point. More importantly, it clearly shows that the high nonlinearity leads to severely oscillating property in terms of the relative error. Actually, our computational experiences tell us that an inappropriate initial point far away a local solution may lead to divergence (see the right plot in Fig. \ref{fig1}), which also implies that designing an implementable and stable algorithm for THDEiCP is a challengeable task. To further verify the ability and reliability of our algorithm, we randomly generate $100$ different starting points such that $u^{(0)}=v^{(0)}$ and their entries are uniformly distributed in $(0,1)$. As we have proved in Theorem \ref{EqEF}, a solution $u^{(k)}$ of THDEiCP must be a nonzero vector. In our experiments, we observe that the algorithm terminates at a zero point in some cases. Accordingly, we record all results of the $100$ tests and divide them into three groups: the first group corresponds to the divergent cases, which means the number of iterations exceeds the preset maximum iterations $20000$; the second group refers to convergent cases but failed to find a nonzero solution of THDEiCP; the last group contains the cases of successfully finding a nonzero solution of THDEiCP. The corresponding rate of each group is graphically shown by the right plot in Fig. \ref{fig1}, which empirically exhibits the ability and reliability of our proposed algorithm. Indeed, we did a lot of experiments on QEiCP, and interestingly, all numerical results shows that the proposed linearized ADMM is always convergent for QEiCP. Thus, such an algorithm further enriches the solvers tailored for QEiCP. This also leaves us an open problem of whether there exists provable global or local convergence for the linearized ADMM on solving the nonconvex model \eqref{mini}.

In Table \ref{table1}, we list several groups of numerical results including starting points $u^{(0)}$, eigenvalue (EigVal $\lambda$), eigenvector ($x$), dual variable ${\bm \varrho}$, number of iterations (`Iter.') and computing time in seconds (Time), where the dual variable is defined by
$${\bm \varrho}:=\lambda^m\A x^{m-1}+\lambda\Bc x^{m-1}-\I x^{m-1},$$
which together with the eigenvector $x$ satisfies $x^\t {\bm \varrho}=0$.

\setlength\rotFPtop{0pt plus 1fil}
\begin{sidewaystable}
\begin{center}
\caption{Computational results by the proposed algorithm.}\vskip 0.2mm
\label{table1}
%\footnotesize
\def\temptablewidth{1\textwidth}
\begin{tabular*}{\temptablewidth}{@{\extracolsep{\fill}}llllllll}\toprule
Example&Initial Points ($u^{(0)}=v^{(0)}$) & EigVal ($\lambda$) & Eigenvector ($x$) & Dual variable (${\bm\varrho}$) & Iter. & Time\\\midrule
Exam. \ref{exam1}& $(0.3829,0.0846,0.7339,0.3320)^\t$ & 0.6830 & $(0.0000,0.0000,0.5701,0.0000)^\t$ & $(0.5042,0.2393,0.0000,0.4162)^\t$ & 326 & 0.32  \\
& $(0.8397,0.3717,0.8282,0.1765)^\t$ & 1.6563 & $(1.2973,0.0000,0.0000,0.0000)^\t$ & $(0.0000,3.0929,4.3342,4.3419)^\t$ & 696 & 0.51  \\
& $(0.1295,0.8799,0.0441,0.6867)^\t$ & 0.8392 & $(0.0000,0.6509,0.0000,0.0000)^\t$ & $(0.4795,0.0000,0.3458,0.5691)^\t$ & 464 & 0.33  \\
& $(0.7338,0.4372,0.3798,0.9797)^\t$ & 1.0561 & $(0.0000,0.0000,0.0000,0.9032)^\t$ & $(1.4153,1.1163,1.2963,0.0000)^\t$ & 576 & 0.45  \\ \hline
Exam. \ref{exam2}& $(0.4030,0.5100,0.4956,0.6514)^\t$ & 0.3947 & $(0.0000,0.0000,0.0000,0.4350)^\t$ & $(0.1242,0.1878,0.1057,0.0000)^\t$ & 1459 & 1.37  \\
& $(0.7437,0.3020,0.0896,0.8260)^\t$ & 0.4747 & $(0.5310,0.0000,0.0000,0.0000)^\t$ & $(0.0000,0.2420,0.2748,0.2551)^\t$ & 1580 & 1.44  \\
& $(0.3896,0.7753,0.1794,0.1094)^\t$ & 0.3528 & $(0.0000,0.3497,0.0000,0.0000)^\t$ & $(0.0745,0.0000,0.0577,0.0866)^\t$ & 944 & 0.86  \\
& $(0.0369,0.5447,0.9976,0.5110)^\t$ & 0.3655 & $(0.0000,0.0000,0.3948,0.0000)^\t$ & $(0.1140,0.1219,0.0000,0.0661)^\t$ & 1481 & 1.36  \\ \hline
Exam. \ref{exam3}& $(0.7919,0.4522,0.8492)^\t$ & 1.2462 & $(0.0000,0.0000,1.1968)^\t$ & $(4.8039,3.6034,0.0000)^\t$ & 743 & 0.72  \\
& $(0.5233,0.4299,0.2072)^\t$ & 0.8860 & $(0.9628,0.0000,0.0000)^\t$ & $(0.0000,0.4632,0.3074)^\t$ & 1234 & 1.12  \\
& $(0.1203,0.6255,0.3466)^\t$ & 0.9807 & $(0.0000,1.0863,0.0000)^\t$ & $(0.9594,0.0000,1.1045)^\t$ & 897 & 0.83  \\\bottomrule
\end{tabular*}
\end{center}
\end{sidewaystable}

It can be easily seen from the data in Table \ref{table1} that our proposed algorithm is fast and reliable for solving the model under consideration, even it is not necessarily convergent for some cases. We conducted many simulations on random data and observed that an appropriate initial point and the parameters, especially the two $\gamma_1$ and $\gamma_2$ are very important for convergence. Therefore, we will pay our attention on the study of convergence results of the iterative scheme \eqref{LADM} in the future. Additionally, we will provide some practical suggestions on choices of $\gamma_1$ and $\gamma_2$.

\section{Conclusions}\label{Concl}
This paper considers a unified model of THDEiCP including TGEiCP and QEiCP as its special cases. As the first work on finding higher-degree cone eigenvalues of tensors, we analyze some corresponding topological properties including closedness, boundedness, and upper-semicontinuity of the $m$-degree $K$-spectrum of ${\mathcal Q}$ (i.e., $\sigma({\mathcal Q},K)$ defined by \eqref{sigmaK}), and the number of $m$-degree Pareto- and $K$-eigenvalues of THDEiCP. For the special case where the underlying tensors $\A$ and $\Bc$ are symmetric, $\Cc=-\I$, and $K:=\R^n_+$, we present a weakly coupled optimization formulation for THDEiCP, which is also a new formulation for QEiCP in the literature. Moreover, such a formulation could bring some numerical benefits for algorithmic design, for instance, an implementable linearized ADMM is developed in this paper. Theoretically, we establish results concerning existence of solutions of THDEiCP with general (symmetric and asymmetric) $\A$ and $\Bc$ when $\Cc:=\pm\I$ and $K:=\R^n_+$. In the future, we will study along this line, but with general tensor $\Cc$ not being a unit tensor (i.e., $\Cc\neq \pm\I$). Of course, the convergence analysis of the proposed algorithm and
designing algorithms for THDEiCP, especially for asymmetric cases, are also our future concerns.

\begin{acknowledgements}
The first two authors would like to thank Professor Deren Han for his comments on the numerical algorithm. This research of the first two authors was supported in part by National Natural Science Foundation of China (11171083, 11301123) and the Zhejiang Provincial NSF (LZ14A010003). The third author was supported by the Hong Kong Research Grant Council (Grant Nos. PolyU 502111, 501212, 501913, and 15302114).
\end{acknowledgements}

%\bibliographystyle{spmpsci}
%\bibliography{E:/Research/JabRef/HeArt,E:/Research/JabRef/Polynomial,E:/Research/JabRef/HeBook}
%

\end{document}